\documentclass{amsart}
\usepackage{amsfonts,graphicx, amsthm,amsmath,amssymb}

\usepackage{hyperref}
\usepackage{color}
\usepackage{enumerate}
\usepackage{enumitem}

\usepackage{todonotes}

\usepackage{bbold} 

\newcommand{\N}{\mathbb{N}}

\newcommand{\R}{\mathbb{R}}

\newcommand{\p}{\mathbb{P}}
\newcommand{\E}{\mathbb{E}}

\newcommand{\En}{\mathcal{E}}
\newcommand{\tEn}{\widetilde{\En}}

\newcommand{\mM}{\mathcal{M}}

\newcommand{\bv}{\overline{v}}
\newcommand{\bg}{\overline{g}}
\newcommand{\bK}{\overline{K}}
\newcommand{\bA}{\overline{A}}

\newcommand{\TV}{\mathrm{TV}}
\newcommand{\vol}{\mathrm{vol}}
\newcommand{\Jac}{\mathrm{Jac}}
\newcommand{\Diff}{\mathrm{Diff}}
\newcommand{\Hol}{\mathrm{Hol}}
\newcommand{\Homeo}{\mathrm{Homeo}}
\newcommand{\Leb}{\mathrm{Leb}}
\newcommand{\supp}{\mathop{\mathrm{supp}}}
\newcommand{\Lip}{\mathop{\mathrm{Lip}}}
\newcommand{\fL}{\mathfrak{L}}
\newcommand{\br}{\bar{r}}
\newcommand{\bx}{\bar{x}}
\newcommand{\by}{\bar{y}}
\newcommand{\bz}{\bar{z}}
\newcommand{\dd}{\, \mbox{d}}

\newcommand{\SL}{\mathrm{SL}}

\newcommand{\mgr}{\mu}
\newcommand{\msp}{\nu}

\newcommand{\dm}{m}

\newcommand{\mTt}{\widehat{\Theta}}
\newcommand{\mTtn}{\widehat{\theta}}
\newcommand{\mTf}{\Theta'}
\newcommand{\mTfn}{\theta'}

\newcommand{\mTn}{\theta}
\newcommand{\dens}{\rho}
\newcommand{\mT}{\Theta}
\newcommand{\tdelta}{\tilde{\delta}}
\newcommand{\dn}{\delta_E}
\newcommand{\Cn}{C_E}
\newcommand{\tC}{\widetilde{C}}

\newcommand{\Ind}{\mathrm{1\!\!I}}

\newcommand{\eps}{\varepsilon}

\newcommand{\diam}{\mathrm{diam}}

\newtheorem{theorem}{Theorem}[subsection]
\newtheorem{lemma}[theorem]{Lemma}

\newtheorem{proposition}[theorem]{Proposition}
\newtheorem{corollary}[theorem]{Corollary}

\theoremstyle{definition}
\newtheorem{definition}[theorem]{Definition}
\newtheorem{remark}[theorem]{Remark}

\makeatletter
\renewcommand\thesubsection{%
  \ifnum\c@subsection=0 \thesection.1%
  \else \thesection.\arabic{subsection}%
  \fi}
\makeatother

\numberwithin{equation}{subsection}

\begin{document}

\title{Log-H\"older regularity of stationary measures}

\date{\today}

\author{Grigorii Monakov}

\address{Grigorii Monakov, Department of Mathematics, University of California, Irvine, CA~92697, USA}

\email{gmonakov@uci.edu}

\begin{abstract}

    We consider Lipschitz and H\"{o}lder continuous random dynamical systems defined by a distribution with a finite logarithmic moment. We prove that under suitable non-degeneracy conditions every stationary measure must be $\log$-H\"{o}lder continuous.

\end{abstract}

\maketitle


\section{Introduction}

\subsection{Outline of the main results}

H\"{o}lder continuity is an important concept that appears naturally in various fields, which include but are not limited to: stochastic processes (see, for example, \cite{MP}), dynamical systems (see \cite{BP, BrK, BV, BBS} and references therein), and spectral theory (e.g. see \cite{CS, DG, GS, HV, Mu, V}).

Recently H\"{o}lder continuity was established for stationary measures of random dynamical systems under very general assumptions:

\begin{theorem}[Gorodetski, Kleptsyn, M., \cite{GKM}] \label{thm:hold}
    Let $M$ be a closed Riemannian manifold. Suppose that $\mgr$ is a probability distribution on $\Diff^1(M)$ such that $\int \|f\|_{\Diff^1}^\gamma d\mgr(f) < \infty$ for some $\gamma > 0$. Suppose also that there is no probability measure $m$ on the manifold $M$ invariant under every map $f \in \supp(\mgr)$. Then every stationary measure of a random dynamical system defined by the distribution $\mgr$ is H\"older continuous.
\end{theorem}

This theorem gives rise to a series of natural questions, that are known to be important in that area. Namely, one can ask
\begin{itemize}
    
    \item What do we get if we impose a weaker condition on the tails of distribution~$\mgr$? For instance, if we only assume that $\int (\log \|f\|_{\Diff^1})^\gamma d\mgr(f) < \infty$.

    \item Is it possible to weaken the regularity assumption on the dynamical system? For example, what if we only assume that homeomorphisms $f \in \supp(\mgr)$ are H\"{o}lder continuous and not differentiable?

\end{itemize}

The aim of present paper is to address these questions. Our main results are the following two theorems (for rigorous formulations and further discussion see Section~\ref{sec:results}):

\begin{theorem}
    Let $\mgr$ be a probability measure on $\Lip(M)$ -- space of bi-Lipschitz homeomorphisms of $M$, satisfying the following assumptions:
    \begin{itemize}
        \item Lipschitz constant has finite $\alpha$-logarithmic moment with respect to $\mgr$.
        \item There is no measure $m$ on $M$, such that $f_* m = m$ for $\mgr$-a.e.~$f$.
    \end{itemize}
    Then every $\mgr$-stationary measure is $\alpha/2$-$\log$-H\"older.
\end{theorem}

\begin{theorem}
    Let $\mgr$ be a probability measure on $\Hol(M)$ -- space of bi-H\"{o}lder homeomorphisms of $M$, satisfying the following assumptions:
    \begin{itemize}
        \item H\"{o}lder constant has a finite logarithmic moment with respect to $\mgr$.
        \item There is no measure $m$ on $M$, such that $f_* m = m$ for $\mgr$-a.e.~$f$.
    \end{itemize}
    Then every $\mgr$-stationary measure is $\log$-H\"older.
\end{theorem}

\subsection{Applications} \label{subsection:applications}

Stationary measures of random dynamical systems are analogues of invariant measures of  deterministic maps, and their properties are of crucial importance for many results in random dynamics, see \cite{A, BH, BQ, BL, Fu, Fu1, Kif1, Kif2, LQ, M} and references therein.

Various results regarding regularity of stationary measures are known. A well studied special case is action of random projective maps on a projective space.

Let $A_1, A_2, \ldots$ be i.i.d. random matrices from $\SL(k, \R)$, distributed with respect to measure $\mgr$. Every $A \in \SL(k, \R)$ induces a projective map $f_{A}: \mathbb{RP}^{k-1}\to \mathbb{RP}^{k-1}$. Denote by $S_{\mgr}$ the closed semigroup in $\SL(k, \R)$ generated by matrices from $\supp(\mgr)$. Under certain irreducibility assumptions (which we won't define rigorously since they are beyond the scope of our paper) on $S_{\mgr}$ Guivarc'h was able to prove the following regularity result for the stationary measure:

\begin{theorem}[Guivarc’h, \cite{Gu}]\label{t.g}
    Suppose that, in the setting above, $S_{\mgr}$ is strongly irreducible and proximal, and
    \begin{equation*}
        \E_{\mgr} \|A_1\|^{\gamma} = \int_{\SL(k, \R)} \|A\|^{\gamma} \dd \mgr(A) < \infty
    \end{equation*}
    for some $\gamma > 0$. Then the corresponding random dynamical system on $\mathbb{RP}^{k-1}$ has unique stationary measure $\msp$ and $\msp$ is H\"older continuous.
\end{theorem}

A regularity estimate similar to the one in Theorem \ref{t.g} is one of the key ingredients in the proof of the following Central Limit Theorem for matrix products:

\begin{theorem}[Le Page, \cite{L}, Guivarc'h, Raugi, \cite{GR}, Gol'dshe\u{\i}d,  Margulis, \cite{GM}] \label{t.LP}
    Let $\{A_k, k \ge 1\}$ be independent and identically distributed random matrices in $SL(k, \R)$, distributed with respect to measure $\mgr$. Assume that $S_{\mgr}$ is \textit{proximal} and \textit{strongly irreducible} and
    \begin{equation} \label{ExpMom}
        \E_{\mgr} \|A_1\|^{\gamma} = \int_{\SL(k, \R)} \|A\|^{\gamma} \dd \mgr(A) < \infty
    \end{equation}
    for some $\gamma > 0$. Then there exists $a > 0$ such that the random variables
    \begin{equation*}
        \frac{\log \|A_n\ldots A_1\|-n\lambda_F}{\sqrt{n}},
    \end{equation*}
    where $\lambda_F > 0$ is the Lyapunov exponent, converge in distribution to $\mathcal{N}(0,a^2)$.
\end{theorem}

It is known that assumption (\ref{ExpMom}), which is usually referred to as \textit{finite exponential moment}, is not optimal for the Central Limit Theorem. Recently Benuist and Quint were able to improve it in the following way: 

\begin{theorem}[Benuist, Quint, \cite{BQ16}]
    Let $\{A_k, k \ge 1\}$ be independent and identically distributed random matrices in $SL(k, \R)$, distributed with respect to measure $\mgr$. Assume that $S_{\mgr}$ is \textit{proximal} and \textit{strongly irreducible} and
    \begin{equation} \label{LogMom}
        \E_{\mgr} \left[ (\log \|A_1\|)^2 \right] = \int_{\SL(k, \R)} (\log \|A\|)^2 \dd \mgr(A) < \infty.
    \end{equation}
    Then there exists $a > 0$ such that the random variables
    \begin{equation*}
        \frac{\log \|A_n\ldots A_1\|-n\lambda_F}{\sqrt{n}},
    \end{equation*}
    where $\lambda_F > 0$ is the Lyapunov exponent, converge in distribution to $\mathcal{N}(0,a^2)$.
\end{theorem}

Condition (\ref{LogMom}), which is usually referred to as \textit{finite second moment} is known to be optimal. In order to establish this more general result (as well as other limit theorems in the same setting) the authors proved a form of log-H\"{o}lder regularity of the stationary measure (see \cite[Proposition 4.5]{BQ16} for details) under \textit{finite second moment} assumption. This result illustrates the importance of log-H\"{o}lder continuity for proving limit theorems with optimal assumptions. At the end of the paper (see Appendix \ref{appendix:B}) we provide a slightly more precise and technical result (in the spirit of $\log$-H\"older continuity) about interaction of two measures generated by non-stationary random dynamical systems. That statement allows us to prove a non-stationary Central Limit Theorem for products of random $\SL(2, \R)$ matrices with optimal assumption on the tails (finite $2 + \eps$ moment for any $\eps > 0$). The proof will be included in the next revision of \cite{GKM2}. We are grateful to the anonymous reviewer for pointing out an unnecessary overestimate in the first version of this work. The resulting improvement was the last missing piece in proving optimal non-stationary Central Limit Theorem.


Log-H\"older regular measures appears naturally in other problems as well. For example, in \cite{CS} Craig and Simon proved that integrated density of states for discrete Schr\"{o}dinger operators with ergodic potentials is $\log$-H\"{o}lder. Curiously, the assumption they found sufficient was a finite logarithmic moment of the function that generates the potential, which is very similar to assumptions (\ref{LogMom}) and (\ref{RegCondLip}). Regularity of integrated density of states of discrete Schr\"{o}dinger operators was extensively studied by many authors, see \cite{B, BK, CK, DG, GK, GS, HV, MS, Mu, ST, V} and references therein.

The paper is structured as follows: in Section \ref{sec:results} we formulate the main results, Section \ref{sec:HolProof} contains a proof of Theorem \ref{thm:mainHol-3}, Section \ref{sec:LipProof} outlines the proof of Theorem \ref{thm:mainLip-3}, Appendix \ref{appendix} collects computational lemmata necessary for the proofs, and Appendix \ref{appendix:B} contains the theorem  about joint regularity of two measures mentioned above.



\section{Main results} \label{sec:results}

Let $M$ be a smooth closed Riemannian  manifold of dimension $k$ and $\mgr$ be a Borel probability measure on $\Homeo(M)$, the set of continuous homeomorphisms of~$M$. Consider the corresponding random dynamical system, given by the compositions \[
T_{n}:=f_{n}\circ\dots\circ f_1,
\]
where $f_i\in \Homeo(M)$ are chosen randomly and independently, with respect to the distribution $\mgr$.

If an initial point $x \in M$ is distributed with respect to a probability measure $\msp$, one can consider the distribution $\mgr * \msp$ of its random image $f(x)$. In other words, $\mgr * \msp$ is the $\mgr$-averaged push-forward image of the measure $\msp$:
\[
    \mgr * \msp := \int_{\Homeo(M)} (f_* \msp) \, \dd \mgr(f).
\]

The measure $\msp$ is called \emph{$\mgr$-stationary} if $\mgr * \msp = \msp$. Let $\mM$ denote the space of all Borel probability measures on $M$ equipped with the weak-* topology.

\begin{definition}
    A measure $\msp \in \mM$ is \textit{$(C,\alpha)$-H\"older} if
    \[
        \forall x\in X \quad \forall r>0 \quad \msp(B_r(x)) \le Cr^{\alpha}.
    \]

    A measure $\msp \in \mM$ is \textit{$(C, \alpha)$-$\log$-H\"older} if
    \[
        \forall x\in X \quad \forall r>0 \quad \msp(B_r(x)) \le C |\log(r)|^{-\alpha}.
    \]
\end{definition}

Let $\Hol(M)$ denote the space of all bi-H\"older continuous homeomorphisms of $M$ and $\Lip(M)$ denote the space of all bi-Lipschitz continuous homeomorphisms of $M$. We define H\"older and Lipschitz constants by the following formulae:

\begin{definition}
    For a map $f:M \to M$ and $\gamma > 0$, let
    \begin{equation*}
        H_{\gamma} (f) = \sup_{\substack{x, y \in M \\ x \ne y}} \frac{d(f(x), f(y))}{d(x, y)^{\gamma}}
    \end{equation*}
    be the (possibly infinite) $\gamma$-H\"older norm of this map, and if $f$ is a homeomorphism, let
    \begin{equation*}
        H'_{\gamma} (f) = \max \left( H_{\gamma} (f), H_{\gamma} \left( f^{-1} \right) \right).
    \end{equation*}
    Next, for $f \in \Hol(M)$ let 
    \begin{equation*}
        \gamma(f) = \sup\{\gamma > 0 \;: \; H'_{\gamma} (f) < \infty\}
    \end{equation*}
    be the critical H\"older regularity for $f$ (that might not be attained), and define
    \begin{equation*}
        L(f) = H'_{\gamma(f)/2} (f)
    \end{equation*}
    be the H\"older constant corresponding to the regularity $\gamma(f)/2$, that is automatically finite. We also introduce the quantity
    \begin{equation*}
        \Lambda(f) = 2 \cdot \frac{1 + \log(L(f))}{\gamma(f)}
    \end{equation*}
    as a (convenient for the future computations) measure of (ir)regularity of a bi-H\"older homeomorphism. Finally, for a bi-Lipschitz $f \in \Lip(M)$ we denote the corresponding bi-Lipschitz constant by
    \begin{equation*}
        \fL(f) = H'_1 (f).
    \end{equation*}
    
\end{definition}

Now we are ready to formulate our main results.

\begin{theorem} \label{thm:mainHol-1}
    Let $\mgr$ be a probability measure on $\Hol(M)$, satisfying the following assumptions:
    \begin{itemize}
        \item 
        There exists $\beta > 0$ such that
        \begin{equation*}
            \int_{\Hol(M)} \Lambda(f)^{\beta} \; \dd \mgr(f) < \infty.
        \end{equation*}
        \item \textbf{(no invariant measure)} There is no measure $m \in \mM$, such that $f_*m=m$ for $\mgr$-a.e.~$f$.
    \end{itemize}
    Then there exist $\alpha>0$ and $C$ such that every $\mgr$-stationary probability measure $\msp$ on $M$ is $(\alpha,C)$-$\log$-H\"older.
\end{theorem}

Moreover, we prove the following regularity estimate for a nonstationary random dynamical systems after finitely many iterations:

\begin{theorem} \label{thm:mainHol-3}
    Let $K$ be a compact set (with respect to weak-* topology) in the space of Borel probability measures on $\Homeo(M)$, satisfying the following assumptions:
    \begin{itemize}
        \item For every $\mgr \in K$ we have $\supp(\mgr) \in \Hol(M)$.
        \item 
        There exists $\beta > 0$ and $C_0$ such that for every $\mgr \in K$
        \begin{equation} \label{RegCondHol}
            \int_{\Hol(M)} \Lambda(f)^{\beta} \; \dd \mgr(f) < C_0.
        \end{equation}
        \item \textbf{(no deterministic image)} For every $\mgr \in K$ there are no measures $m_1, m_2 \in \mM$, such that $f_*m_1=m_2$ for $\mgr$-a.e.~$f$.
    \end{itemize}
    Then there exist $\alpha > 0$, $C$ and $\kappa > 1$ such that for every initial measure $\msp_0$, every sequence of measures $\mgr_1, \mgr_2, \ldots \in K $, every number of iterations $n \in \N$, and every $x\in M$ one has:
    \begin{equation*}
        \text{if}\ \  r>e^{-\kappa^n} \ \ \text{then}\ \ [\mgr_n * \mgr_{n - 1} * \ldots \mgr_1 * \msp_0](B_r(x)) < C |\log(r)|^{-\alpha}.
    \end{equation*}
\end{theorem}

We would like to mention the following connection between the \textbf{no invariant measure} and the \textbf{no deterministic image} assumptions. 

\begin{proposition} \label{prop:conv}
    Let $\mgr$ be a probability measure on $\Homeo (M)$ that satisfies \textbf{no invariant measure} condition (there is no common invariant measure for all $f \in \text{supp}\,\mgr$). Then there exists $k \in \mathbb{N}$ such that $\mgr^{*k}$ satisfies \textbf{no deterministic images} condition (i.e. there are no probability measures $\msp, \msp'$ on~$M$ such that $f_* \msp = \msp'$ with $f = f_k \circ f_{k - 1} \circ \ldots \circ f_1$ for $\mgr \times \mgr \times \ldots \times \mgr$-almost all $(f_1, f_2, \ldots, f_k) \in \left(\Homeo (M) \right)^k$).
\end{proposition}

Indeed, if for any $k \in \N$ there exists a measure with the same image under every $f \in \supp(\mgr^k)$ then any accumulation point of Krylov-Bogolyubov time averages of these measures has to be an invariant measure (see \cite{GKM} for more details).

Applying Theorem \ref{thm:mainHol-3} to a one-point set $K = \{ \mgr \}$ and taking into account Proposition \ref{prop:conv} we obtain the following 

\begin{corollary} \label{cor:Hol}
    Assume that $\mgr$ satisfies the assumptions of Theorem~\ref{thm:mainHol-1}. Then there exist $\alpha>0$, $C$ and $\kappa>1$ such that for every initial measure $\msp_0$, every number of iterations $n \in \N$, and every $x\in M$ one has:
    \begin{equation} \label{e.scale1}
        \text{if}\ \  r>e^{-\kappa^n} \ \ \text{then}\ \ [\mgr^{*n}*\msp_0](B_r(x)) < C |\log(r)|^{-\alpha}.
    \end{equation} 
\end{corollary}

\begin{remark} \label{rem:Main}
    It is also easy to see that Theorem \ref{thm:mainHol-1} follows from Corollary \ref{cor:Hol}, hence all we need to prove is Theorem \ref{thm:mainHol-3}.
\end{remark}

Using similar techniques we are also able to provide more refined estimates for the case of bi-Lipschitz homeomorphisms.

\begin{theorem} \label{thm:mainLip-1}
    Let $\mgr$ be a probability measure on $\Lip(M)$, satisfying the following assumptions:
    \begin{itemize}
        \item 
        There exists $\alpha > 0$ such that
        \begin{equation} \label{RegCondLip}
            \int_{\Lip(M)} (\log(\fL(f)))^{\alpha} \; d\mgr(f) < \infty.
        \end{equation}
        \item 
        There is no measure $m \in \mM$, such that $f_*m=m$ for $\mgr$-a.e.~$f$.
    \end{itemize}
    Then there exists $C$, such that every stationary measure $\msp \in \mM$ is $(C, \alpha/2)$-$\log$-H\"older.
\end{theorem}


Similar to the case of H\"older continuous homeomorphisms, a nonstationary version of Theorem \ref{thm:mainLip-1} is available.

\begin{theorem} \label{thm:mainLip-3}
    Let $K$ be a compact set (with respect to weak-* topology) in the space of Borel probability measures on $\Homeo(M)$, satisfying the following assumptions:
    \begin{itemize}
        \item For every $\mgr \in K$ we have $\supp(\mgr) \subset \Lip(M)$.
        \item 
        There exists $\alpha > 0$ and $C_0$ such that for every $\mgr \in K$
        \begin{equation} \label{AssumptionLip}
            \int_{\Lip(M)} \left( \log(\fL(f)) \right)^{\alpha} \; d\mgr(f) < C_0.
        \end{equation}
        \item 
        For every $\mgr \in K$ there are no measures $m_1, m_2 \in \mM$, such that $f_*m_1=m_2$ for $\mgr$-a.e.~$f$.
    \end{itemize}
    Then there exist $C$ and $\kappa > 1$ such that for every initial measure $\msp_0$, every sequence of measures $\mgr_1, \mgr_2, \ldots \in K$, every number of iterations $n \in \N$, and every $x\in M$ one has:
    \begin{equation*}
        \text{if}\ \  r>e^{-\kappa^n} \ \ \text{then}\ \ [\mgr_n * \mgr_{n - 1} * \ldots \mgr_1 * \msp_0](B_r(x)) < C |\log(r)|^{-\frac{\alpha}{2}}.
    \end{equation*}
\end{theorem}

The following corollary also holds:
\begin{corollary} 
    Assume that $\mgr$ satisfies the assumptions of Theorem~\ref{thm:mainLip-1}. Then there exist $C$ and $\kappa > 1$ such that for every initial measure $\msp_0$, every number of iterations $n \in \N$, and every $x \in M$ one has:
    \begin{equation*}
        \text{if}\ \  r>e^{-\kappa^n} \ \ \text{then}\ \ [\mgr^{*n}*\msp_0](B_r(x)) < C |\log(r)|^{-\frac{\alpha}{2}}.
    \end{equation*}
\end{corollary}

\section{H\"older continuous homeomorphisms} \label{sec:HolProof}

\subsection{Outline of the proof} \label{subsection:plan}

This section is devoted to the proof of Theorem \ref{thm:mainHol-3}. As it was already mentioned in Corollary \ref{cor:Hol} and Remark \ref{rem:Main}, it will immediately imply Theorem \ref{thm:mainHol-1}. The regularity of a measure $\msp$ on $M$ can be established by considering the integral
\begin{equation} \label{IntForInt}
    \iint_{M^2} |\log(d(x, z))|^{\alpha} \dd \msp(x) \dd \msp(z).
\end{equation}
Namely, if it is finite, then the measure of any ball $B_r(x)$ by the Markov inequality doesn't exceed
\begin{equation*}
    \msp(B_r(x)) \le c |\log(r)|^{-\frac{\alpha}{2}}.
\end{equation*}
Moreover, a similar estimate can be obtained given finiteness of any integral
\begin{equation} \label{EnInt}
    \En_{U}(\msp) = \iint_{M \times M} U(d(x, z)) \dd \msp(x) \dd\msp(z),
\end{equation}
where the function $U(r)$ has the same singularity as $|\log(r)|^{\alpha}$ as $r \to 0$. Indeed, if a lower bound $U(r) > c |\log(r)|^{\alpha}$ holds for some $r > \eps$, then for the same $r$ one will have the bound for the measures $\msp(B_r(x))$.

To estimate integrals of the type \eqref{IntForInt}, we will apply the machinery introduced in \cite{GKM}. Namely, we will consider the quantity $\En_{\alpha, \eps} (\msp) = \En_{U_{\alpha, \eps}} (\msp)$, associated to some function $U_{\alpha, \eps}$, and satisfying the following properties:

\begin{enumerate}[label=(\roman*)]

    \item\label{i} For a fixed $\alpha > 0$ it admits a uniform upper bound, depending only on the choice of the minimal radius $\eps$:
    \begin{equation} \label{upperInt}
        \En_{\alpha, \eps} (\msp) \le U_{\alpha, \eps} (0)
    \end{equation}
    (see Lemma \ref{lm:EnUpperBound}).
    
    \item\label{ii} Outside the radius $\eps$ it satisfies the lower bound 
    \begin{equation} \label{lowerInt}
        U_{\alpha, \eps} (r) \ge c |\log(r)|^{\alpha}
    \end{equation}
    (see Proposition \ref{prop:UAEps} part \textbf{(III)}).
    
    \item\label{iii} A convolution with any $\mgr \in K$ reduces the quantity $\En_{\alpha, \eps}$ by a linear factor: for some $\lambda < 1$, $\tilde{C}$ for any $\eps > 0$, any $\msp$ and any $\mgr \in K$ one has
    \begin{equation} \label{contrInt}
        \En_{\alpha, \eps} (\mgr * \msp) < \max(\lambda \En_{\alpha, \eps} (\msp), \tilde{C}) 
    \end{equation}
    (see Proposition \ref{prop:contr} and Corollary \ref{c:En}).
    
\end{enumerate}

Having constructed such functions, for every $n$ we choose $\eps$ so that the upper bound
\begin{equation} \label{unifUpperInt}
    \En_{\alpha, \eps} (\mgr_n * \mgr_{n - 1} * \ldots * \mgr_1 * \msp) < \max(\lambda^n U_{\alpha, \eps} (0), \tilde{C})
\end{equation}
obtained from joining \ref{i} and \ref{iii}, would be equal to $\tilde{C}$. Then, the application of the Markov inequality to \ref{ii} concludes the proof.

The most technically involved part of the proof is establishing inequality \eqref{contrInt}. The key idea is to replace the energy $\En_{\alpha, \eps} (\msp)$ with an equivalent one $\tEn_{\alpha, \eps} (\msp)$ (see Section \ref{subsection:tEn}), where $\tEn_{\alpha, \eps}$ can be represented as a square of the  norm of an $L^2(M, \Leb)$ function. Let us call this function $\rho_{\alpha, \eps}[\msp] (y)$. It will follow from the construction that $$\rho_{\alpha, \eps}[\mgr * \msp](y) = \E_{\mgr} \rho_{\alpha, \eps}[f_* \msp](y),$$ and that functions $\rho_{\alpha, \eps}[f_* \msp](y)$ are comparable in $L^2$ norm with $\rho_{\alpha, \eps}[\msp] (y)$ for bi-H\"older $f$. Now the idea of the proof can be formulated as follows: the energy $\tEn_{\alpha, \eps} (\mgr * \msp)$ is a square of norm of a convex combination of $L^2(M, \Leb)$ vectors of norm comparable to the one of $\rho_{\alpha, \eps}[\msp]$:
\begin{equation*}
    \tEn_{\alpha, \eps} (\mgr * \msp) = \| \E_{\mgr} \rho_{\alpha, \eps}[f_* \msp] \|^2_{L^2}, \quad \text{and} \quad \|\rho_{\alpha, \eps}[f_* \msp]\|_{L^2} \sim \|\rho_{\alpha, \eps}[\msp]\|_{L^2}.
\end{equation*}
If the contraction doesn't happen (inequality \eqref{contrInt} doesn't hold), then the convex combination has a norm that is comparable to that of its components. This fact implies that the components are aligned with each other. In other words, the functions $\rho_{\alpha, \eps}[f_* \msp](y)$ are close to one another for $\mgr$-almost all $f$. Hence, so are the measures given by 
\begin{equation*}
    \theta_{\alpha, \eps} [f_* \msp] = \frac{\rho_{\alpha, \eps}^2 [f_* \msp] (y) \dd \Leb(y)}{\tEn_{\alpha, \eps} (f_* \msp)}
\end{equation*}
(notice that due to our choice of $\tEn_{\alpha, \eps}(\msp)$ and $\rho_{\alpha, \eps} [\msp]$ the measure $\theta_{\alpha, \eps} [\msp]$ always has total mass equal to 1). Finally, we show that measures $\theta_{\alpha, \eps} [f_* \msp]$ and $f_* \theta_{\alpha, \eps} [\msp]$ are also close for any bi-H\"older $f$. Now it remains to notice that for $\mgr$-almost every $f$ the measures $f_* \theta_{\alpha, \eps} [\msp]$ have to be close to one another, which turns the measure $\theta_{\alpha, \eps} [\msp]$ into a measure with an ``almost deterministic image''. After passing to a limit, we obtain a true measure with deterministic image and hence arrive to a contradiction.

The plan for the rest of of this section is as follows:
\begin{itemize}
    \item In Section \ref{subsection:Uphi} we define functions $U_{\alpha, \eps}$ and $\varphi_{\alpha, \eps}$ (which are needed to construct $\En_{\alpha, \eps}$ and $\tEn_{\alpha, \eps}$ mentioned above).
    \item In Section \ref{subsection:Uae} we collect some useful properties of $U_{\alpha, \eps}$, including the lower bound \eqref{lowerInt} (Proposition \ref{prop:UAEps} part \textbf{(III)}).
    \item In Section \ref{subsection:En} we introduce $\En_{\alpha, \eps}$. We prove the upper bound \eqref{upperInt} (Lemma \ref{lm:EnUpperBound}), estimates on how $\En_{\alpha, \eps}$ behaves if we replace $\msp$ with $f_* \msp$ for a bi-H\"older $f$ (Proposition \ref{prop:EnChangeHol}), and a Markov's type inequality (Lemma \ref{lm:EnLowerBound}) to be used in the end of the proof of Theorem \ref{thm:mainHol-3}.
    \item In Section \ref{subsection:tEn} we introduce $\tEn_{\alpha, \eps}$, $\rho_{\alpha, \eps}$, and $\theta_{\alpha, \eps}$ mentioned above. We prove that energies $\En_{\alpha, \eps}$ and $\tEn_{\alpha, \eps}$ are equivalent (Proposition \ref{prop:EnEquiv}).
    \item In Section \ref{subsection:theta} we prove that measures $\theta_{\alpha, \eps} [f_* \msp]$ and $f_* \theta_{\alpha, \eps} [\msp]$ are close to each other for any bi-H\"older $f$ (Proposition \ref{prop:WassEst}). We also show that the average (with respect to $\mgr$) value of $\tEn_{\alpha, \eps} (f_* \msp)$ is close to $\tEn_{\alpha, \eps} (\msp)$ (see Proposition \ref{lm:tailsEst}). This is where the tail estimate \eqref{RegCondHol} comes into play.
    \item In Section \ref{subsection:contr} we prove the contraction property \eqref{contrInt} (Proposition \ref{prop:contr}).
    \item In Section \ref{subsection:finalProof} we establish the upper bound \eqref{unifUpperInt} and finish the proof of Theorem \ref{thm:mainHol-3}. 
\end{itemize}

\subsection{Choice of the functions $U_{\alpha, \eps}$ and $\varphi_{\alpha, \eps}$} \label{subsection:Uphi}

As mentioned above, we would like to find a function $\varphi_{\alpha} (r)$, such that $\varphi_{\alpha} * \varphi_{\alpha} (r)$ has a singularity of order $|\log(r)|^{\alpha}$ at the origin. A natural candidate is given by a square root of the derivative of $|\log(r)|^{\alpha}$. As we will be working on a Riemannian manifold of dimension $k$ we adjust $\varphi_{\alpha}$ accordingly:

\begin{definition} \label{def:phia}
    Define
    \begin{equation*}
        \varphi_{\alpha} (x) =
        \begin{cases}
            \frac{( - \log(x))^{\frac{\alpha - 1}{2}}}{x^{\frac{k}{2}}}, \quad &\text{if $0 < x < \frac{1}{e}$,}\\
            0, \quad &\text{otherwise,}
        \end{cases}
    \end{equation*}
    and for $r > 0$
    \begin{multline*}
        U_{\alpha} (r) = \int_{\R^k} \varphi_{\alpha} (|\bx|) \varphi_{\alpha} (|\br - \bx|) \dd \bx = \\
        = \int_{\{|\bx| < \frac{1}{e}\} \cap \{|\br - \bx| < \frac{1}{e}\}} \varphi_{\alpha} (|\bx|) \varphi_{\alpha} (|\br - \bx|) \dd \bx,
    \end{multline*}
    where $\bx \in \R^k$ and $\br = (r, 0, \dots, 0) \in \R^k$.
\end{definition}

The function $U_{\alpha} (r)$ has the same singularity at the origin as $|\log(r)|^{\alpha}$ 
(see Lemma \ref{lm:partInt}), but still needs a little adjustment, because the integral 
\begin{equation*}
    \iint_{M^2} U_{\alpha} (d(x, z)) \dd \msp(x) \dd \msp(z)
\end{equation*}
might be infinite and the estimate mentioned in Part \ref{i} of the outline will be meaningless. To fix that problem we define the following family of cut-offs:

\begin{definition} \label{def:phiae}

    Define
    \begin{equation*}
        \varphi_{\alpha, \eps} (x) = \begin{cases}
            \varphi_{\alpha} (x), \quad \text{if $x \ge \eps$}; \\
            \varphi_{\alpha}(\eps), \quad \text{if $0 < x < \eps$}
        \end{cases}
    \end{equation*}
    and
    \begin{multline} \label{def:UAE}
        U_{\alpha, \eps} (r) = \int_{\R^k} \varphi_{\alpha, \eps} (|\bx|) \varphi_{\alpha, \eps} (|\br - \bx|) \dd \bx = \\
        = \int_{\{|\bx| < \frac{1}{e}\} \cap \{|\br - \bx| < \frac{1}{e}\}} \varphi_{\alpha, \eps} (|\bx|) \varphi_{\alpha, \eps} (|\br - \bx|) \dd \bx,
    \end{multline}
    where $\bx \in \R^k$ and $\br = (r, 0, \dots, 0) \in \R^k$.

\end{definition}

\subsection{Properties of $U_{\alpha, \eps}$} \label{subsection:Uae}
In this section we collect some important properties of $U_{\alpha, \eps}$.

An important symmetry that the function $U_{\alpha, \eps}$ possesses is described in the following

\begin{remark} \label{rem:symm}
    Note that the definition of $U_{\alpha, \eps}$ can be reformulated as follows: for any two points $\bx, \bz \in \R^k$ consider the integral
    \begin{equation} \label{eq:I}
        I_{\alpha, \eps} (\bx, \bz) := \int_{\R^k} \varphi_{\alpha, \eps}(|\by - \bx|) \varphi_{\alpha, \eps}(|\by - \bz|) \, \dd \Leb(\by).
    \end{equation}
    Due to the spherical symmetry, this integral depends only on the distance $r = |\bx - \bz|$ between these two points:
    \begin{equation*}
        I_{\alpha, \eps} (\bx, \bz) = U_{\alpha, \eps}(|\bx - \bz|)
    \end{equation*}
    for some function $U_{\alpha, \eps}$ of $r = |\bx - \bz|$ (which of course has to coincide with $U_{\alpha, \eps}$ defined by \eqref{def:UAE}). We can take this as a definition of~$U_{\alpha, \eps} (r)$.
\end{remark}

Let us denote by $c_{\alpha, k}$ the following constant:
\begin{equation*}
    c_{\alpha, k} = \frac{k \omega_k}{\alpha},
\end{equation*}
where $\omega_k$ it the volume of a $k$-dimensional unit ball. 
The following Proposition establishes the properties of $U_{\alpha, \eps}$ outlined in parts \ref{i} and \ref{ii} of the plan, together with some other estimates that will be useful later.
\begin{proposition} \label{prop:UAEps}
    For every dimension $k \ge 1$ the following holds:
    \begin{enumerate}

        \item[\textbf{(I)}] For every $\alpha > 0$ and $\eps > 0$ the function $U_{\alpha, \eps}$ is non-increasing.

        \item[\textbf{(II)}] For every $\alpha_0 > 0$ there exist $r_0 > 0$ and $\eps_0 > 0$, such that for every $0 < \alpha < \alpha_0$, every $0 < \eps < \eps_0$, and every $0 < r < r_0$ we have:
        \begin{equation} \label{ineq:UB}
            U_{\alpha, \eps} (r) < 2 c_{\alpha, k} (- \log(r))^{\alpha}.
        \end{equation}

        \item[\textbf{(III)}] For every $\alpha_0 > 0$ there exist $\eps_0 > 0$, such that for every $0 < \alpha < \alpha_0$, and every $r, \eps$, such that $0 < \eps < r < \eps_0$ we have:
        \begin{equation} \label{ineq:LB}
            U_{\alpha, \eps} (r) > \frac{c_{\alpha, k}}{2} (- \log(r))^{\alpha}.
        \end{equation}

        \item[\textbf{(IV)}] For every $\alpha_0 > 0$ and every $\delta > 0$ there exist $\eps_0 > 0$ and $r_0 > 0$ such that for every $0 < \alpha < \alpha_0$, every $0 < \eps < \eps_0$, and every $0 < r_1 \le r_2 \le r_0$ we have
        \begin{equation} \label{UOneStep}
            U_{\alpha, \eps} (r_1) \le (1 + \delta) \left( \frac{\log(r_1)}{\log(r_2)} \right)^{\alpha} U_{\alpha, \eps} (r_2).
        \end{equation}

    \end{enumerate}

\end{proposition}

We will postpone the proof of Proposition \ref{prop:UAEps} until Appendix \ref{appendix}. 

\subsection{Definition and properties of $\En_{\alpha, \eps}$: behaviour under images and convolutions} \label{subsection:En}

We are ready to define the energy 
mentioned in \eqref{EnInt}.

\begin{definition} \label{def:En}
    \begin{equation*}
        \En_{\alpha, \eps} (\msp) = \iint_{M^2} U_{\alpha, \eps} (d(x, z)) d \msp(x) d \msp(z).
    \end{equation*}
\end{definition}

Let us point out a few useful properties of $\En_{\alpha, \eps}$. We start with a precise estimate for the energy of an arbitrary measure $\msp$ for given $\alpha$ and $\eps$ (as mentioned in part \ref{i} of the outline):

\begin{lemma} \label{lm:EnUpperBound}
    For every $\alpha > 0$, every $\msp \in \mM$ and every $0 < \eps < 1/e$ we have the following upper bound:
    \begin{equation} \label{EnUpperBound}
        \En_{\alpha, \eps} (\msp) \le (\omega_k + c_{\alpha, k}) (- \log(\eps))^{\alpha}
    \end{equation}
\end{lemma}

\begin{proof}
    A straightforward computation shows that 
    \begin{equation*}
        \En_{\alpha, \eps} (\msp) \le U_{\alpha, \eps} (0) = \int_{0 < |\bar{x}| < \eps} \varphi_{\alpha}^2 (\eps) \dd \bar{x} + \int_{\eps < |\bar{x}| < \frac{1}{e}} \varphi_{\alpha}^2 (|\bar{x}|) \dd \bar{x}.
    \end{equation*}
    Both summands can be computed exactly. For the first one we substitute the definition of $\varphi_{\alpha} (\eps)$:
    \begin{equation*}
        \int_{0 < |\bar{x}| < \eps} \varphi_{\alpha}^2 (\eps) \dd \bar{x} = \frac{\omega_k \eps^k (- \log(\eps))^{\alpha - 1}}{\eps^k} = \omega_k (- \log(\eps))^{\alpha - 1}.
    \end{equation*}
    For the second one we use spherical coordinates (see Lemma \ref{lm:partInt} for details): 
    \begin{equation*}
        \int_{\eps < |\bar{x}| < \frac{1}{e}} \varphi_{\alpha}^2 (|\bar{x}|) \dd \bar{x} = c_{\alpha, k} ((-\log(\eps))^{\alpha} - 1).
    \end{equation*}
    Adding together the two inequalities above we arrive to \eqref{EnUpperBound}.
    
\end{proof}

Next, we show a Markov type inequality that estimates $\msp(B_{r}(x))$ through $\En_{\eps, \alpha} (\msp)$:

\begin{lemma} \label{lm:EnLowerBound}
    For every $\alpha > 0$ there exists $C_{\alpha, k} < \infty$ and $\eps_0 > 0$, such that for every $0 < \eps < \eps_0$, every $\msp \in \mM$, every $0 < \eps < r < \frac{1}{e}$ and every ball $B_{r}(x) \subset M$ the following holds:
    \begin{equation} \label{EnVSlogHol}
        \msp (B_{r}(x)) \le C_{\alpha, k} \sqrt{\En_{\alpha, \eps} (\msp)} (-\log(r))^{-\frac{\alpha}{2}}.
    \end{equation}
\end{lemma}

\begin{proof}
    Let us choose $\eps_0$ such that Part \textbf{(III)} of Proposition \ref{prop:UAEps} is applicable. Then for small enough $r$ applying Markov inequality to the definition of $\En_{\alpha, \eps}$ we obtain:
    \[
        \En_{\alpha,\eps}(\msp) \ge U_{\alpha,\eps} (2r) \cdot \msp(B_{r}(x))^2 \ge \frac{c_{\alpha, k}}{2}(-\log(2r))^{\alpha} \cdot \msp(B_{r}(x))^2.
    \]
    Choosing $C_{\alpha, k}$ big enough we can ensure that (\ref{EnVSlogHol}) holds for all $r$ such that $\eps < r < 1/e$.
\end{proof}

Now we need some preparations in order to describe the energy $\En_{\alpha, \eps} (\mgr * \msp)$ that we get after one step of our random dynamics. The following estimates will be useful for proving the contraction mentioned in part \ref{iii} of the outline. We start with the following series of statements describes the change of the energy $\En_{\alpha, \eps} (\msp)$ after applying a H\"older continuous homeomorphism.

\begin{proposition} \label{prop:EnChangeHol}
    For every $\delta > 0$ there exists $\eps_0 > 0$, such that for every $0 < \alpha < 1$ there exist $A = A (\alpha, \delta)$, such that for every $0 < \eps < \eps_0$, every $f \in \Hol(M)$ and every $\msp \in \mM$ the following holds:
    \begin{equation} \label{EnOneStep}
        \En_{\alpha, \eps} (f_* \msp) \le \Lambda(f)^{\alpha} \left( (1 + \delta) \En_{\alpha, \eps} (\msp) + A \right).
    \end{equation}
\end{proposition}

\begin{proof}

    According to the Part \textbf{(IV)} of Proposition \ref{prop:UAEps} we can choose $\eps_0 > 0$ and $r_0 > 0$, such that for every $0 < \alpha < 1$, every $0 < \eps < \eps_0$ and every $0 < r_1 \le r_2 \le r_0$ we have
    \begin{equation*}
        U_{\alpha, \eps} (r_1) \le (1 + \delta) \left( \frac{\log(r_1)}{\log(r_2)} \right)^{\alpha} U_{\alpha, \eps} (r_2).
    \end{equation*}
    Let us split the integral from the definition of energy the $\En_{\alpha, \eps} (f_* \msp)$ in the following way:
    \begin{multline*}
        \En_{\alpha, \eps} (f_* \msp) = \iint_{M \times M} U_{\alpha, \eps} (d(x, z)) \dd f_* \msp(x) \dd f_* \msp(z) = \\
        = \iint_{M \times M} U_{\alpha, \eps} (d(f(x), f(z))) \dd \msp(x) \dd \msp(z) = \\
        = \iint_{d(x, z) < r_0} U_{\alpha, \eps} (d(f(x), f(z))) \dd \msp(x) \dd \msp(z) + \\
        + \iint_{d(x, z) \ge r_0} U_{\alpha, \eps} (d(f(x), f(z))) \dd \msp(x) \dd \msp(z).
    \end{multline*}
    To estimate the first summand we take $r_2 = d(x, z)$ and 
    $r_1 = \min(d(f(x), f(z)), r_2)$ and we deduce that
    \begin{multline*}
        U_{\alpha, \eps} (d(f(x), f(z))) \le U_{\alpha, \eps} (r_1) \le (1 + \delta) \left( \frac{\log(r_1)}{\log(r_2)} \right)^{\alpha} U_{\alpha, \eps} (r_2) \le \\
        \le (1 + \delta) \left( \frac{\log(d(f(x), f(z)))}{\log(d(x, z))} \right)^{\alpha} U_{\alpha, \eps} (d(x, z))
    \end{multline*}
    for any $x, z \in M$, such that $d(x, z) \le r_0$. By definition of $\gamma(f)$ and $L(f)$ we know that
    \begin{equation*}
        d(f(x), f(z)) \ge \left( \frac{d(x, z)}{L(f)} \right)^{\frac{2}{\gamma(f)}}
    \end{equation*}
    for any $x, z \in M$. After substituting that into previous inequality we arrive to
    \begin{multline*}
        U_{\alpha, \eps} (d(f(x), f(z))) \le \\
        \le (1 + \delta) \left( \frac{2 (\log(d(x, z)) - \log(L(f)))}{\gamma(f) \log(d(x, z))} \right)^{\alpha} U_{\alpha, \eps} (d(x, z)) \le \\
        \le (1 + \delta) \Lambda(f)^{\alpha} U_{\alpha, \eps} (d(x, z)).
    \end{multline*}
    For the first summand the last inequality gives us the following:
    \begin{multline*}
        \iint_{d(x, z) < r_0} U_{\alpha, \eps} (d(f(x), f(z))) \dd \msp(x) \dd \msp(z) \le \\
        \le (1 + \delta) \iint_{M \times M} \Lambda(f)^{\alpha} U_{\alpha, \eps} (d(x, z)) \dd \msp(x) \dd \msp(z) \le \\
        \le (1 + \delta) \Lambda(f)^{\alpha} \En_{\alpha, \eps} (\msp).
    \end{multline*}
    Let us take $r_0$ and $\eps_0$ small enough to ensure that the conculsion of Part \textbf{(II)} of Proposition \ref{prop:UAEps} also holds. That allows us to estimate the second summand as follows:
    \begin{multline*}
        \iint_{d(x, z) \ge r_0} U_{\alpha, \eps} (d(f(x), f(z))) \dd \msp(x) \dd \msp(z) \le \\
        \le \iint_{M \times M} U_{\alpha, \eps} \left( \left( \frac{r_0}{L(f)} \right)^{\frac{2}{\gamma(f)}} \right) \dd \msp(x) \dd \msp(z) \le \\
        \le 2 c_{\alpha, k} \left( - \log \left( \left( \frac{r_0}{L(f)} \right)^{\frac{2}{\gamma(f)}} \right) \right)^{\alpha} \le \\
        \le 2 c_{\alpha, k} (- \log(r_0))^{\alpha} \Lambda(f)^{\alpha}.
    \end{multline*}
    Taking $A = 2 c_{\alpha, k} (- \log(r_0))^{\alpha}$ finishes the proof.
\end{proof}

An immediate corollary can be formulated as follows:

\begin{corollary} \label{cor:EnHol}
    For every $\delta > 0$ there exists $\eps_0 > 0$, such that for every $0 < \alpha < 1$ there exists $C < \infty$, such that for every $0 < \eps < \eps_0$, every $f \in \Hol(M)$ and every $\msp \in \mM$, such that $\En_{\alpha, \eps} (\msp) > C$ the following holds:
    \begin{equation} \label{EnHolOneStep}
        \Lambda(f)^{-\alpha} \frac{1}{1 + \delta}
        \le \frac{\En_{\alpha, \eps} (f_* \msp)}{\En_{\alpha, \eps} (\msp)}
        \le \Lambda(f)^{\alpha} (1 + \delta).
    \end{equation}
\end{corollary}

Another corollary of this statement that would be convenient to use is
\begin{corollary} \label{cor:EnOneStepForm2}
    There exists $\eps_0 > 0$, such that for every $0 < \alpha < 1$ there exists $C < \infty$, such that for every $0 < \eps < \eps_0$, every $f \in \Hol(M)$ and every $\msp \in \mM$ the following holds:
    \begin{equation} \label{e:EnOneStepUnif}
        \En_{\alpha, \eps} (f_* \msp) \le \max \left( 2\Lambda(f)^{\alpha} \En_{\alpha, \eps} (\msp), C \right).
    \end{equation}
\end{corollary}

\begin{proof}
    Applying Corollary \ref{cor:EnHol} to $\delta = \frac{1}{2}$, $\tilde{\msp} := f_* \msp$ and $\tilde{f} = f^{-1}$ we conclude that there exists $C < \infty$, such that if $\En_{\alpha, \eps} (\tilde{\msp}) = \En_{\alpha, \eps} (f_* \msp) > C$ then 
    \[
        \En_{\alpha, \eps} (f_*^{-1} f_* \msp) \ge \frac{1}{1 + \frac{1}{2}} \Lambda \left( f^{-1} \right)^{-\alpha} \En_{\alpha, \eps} (f_* \msp).
    \]
    It remains to recall that $\Lambda(f) = \Lambda(f^{-1})$ and the inequality (\ref{e:EnOneStepUnif}) follows.
\end{proof}

\subsection{Definition of $\tEn_{\alpha, \eps}$ and comparison to $\En_{\alpha, \eps}$} \label{subsection:tEn}
Following the technique from \cite{GKM} we introduce another energy $\tEn_{\alpha, \eps} (\msp)$. The intuition behind it can be described as follows: on one hand, for singular enough $\msp$ (such that $\En_{\alpha, \eps} (\msp)$ is big) the values $\En_{\alpha, \eps} (\msp)$ and $\tEn_{\alpha, \eps} (\msp)$ will be close to each other, so we can use $\tEn_{\alpha, \eps} (\msp)$ to estimate $\En_{\alpha, \eps} (\msp)$. On the other hand, the energy $\tEn_{\alpha, \eps}$ can be represented as a square of the norm of an $L^2(M, \Leb)$ vector, which allows us to use geometry of this Hilbert space when working with $\tEn_{\alpha, \eps}$. This makes $\tEn_{\alpha, \eps}$ a crucial tool for proving the contraction property described in part \ref{iii} of the outline.

We start by defining the corresponding $L^2(M, \Leb)$ vector mentioned above:
\begin{definition} \label{def:dens}
    For $\msp \in \mM$ define
    \begin{equation} \label{eq:dens-a-e}
        \dens_{\alpha, \eps} [\msp] (y) = \int_{M} \varphi_{\alpha, \eps} (d(x, y)) \dd \msp(x).\\
    \end{equation}
    It will be convenient to have a notation for a measure that has density with respect to $\Leb$, which is equal to $\dens_{\alpha, \eps} [\msp]$:
    \begin{equation*}
        \mT_{\alpha, \eps} [\msp] = \rho_{\alpha, \eps}^2 [\msp] (y) \dd \Leb(y).
    \end{equation*}
\end{definition}
Next, we define $\tEn_{\alpha, \eps}$ and a normalized version of the measure $\Theta_{\alpha, \eps}$:
\begin{definition} \label{def:tEn}
    For $\msp \in \mM$ define:
    \begin{equation} \label{eq:tEn}
        \tEn_{\alpha, \eps} (\msp) = \mT_{\alpha, \eps} [\msp] (M);
    \end{equation}
    and
    \begin{equation*}
        \mTn_{\alpha, \eps} [\msp] = \frac{\Theta_{\alpha, \eps} [\msp]}{\tEn_{\alpha, \eps} (\msp)}.
    \end{equation*}
\end{definition}


In order to compare $\tEn_{\alpha, \eps}(\msp)$ to $\En_{\alpha, \eps}(\msp)$ we show that the former can be represented as an integral of a certain kernel over $\msp \times \msp$. Later we will show that this kernel is close to $U_{\alpha, \eps} (d(x, z))$.
\begin{lemma} \label{l:triple}
    \begin{equation*}
        \tEn_{\alpha,\eps}(\msp) = \iint_{M\times M} K_{\alpha,\eps}(x,z) \dd \msp(x) \dd \msp(z),
    \end{equation*}
    where
    \begin{equation} \label{eq:K-a-e}
        K_{\alpha,\eps}(x,z) := \int_M \varphi_{\alpha,\eps}(d(x,y)) \varphi_{\alpha,\eps}(d(z,y)) \, \dd \, \Leb(y).
    \end{equation}
\end{lemma}
\begin{proof}
    It suffices to substitute (\ref{eq:dens-a-e}) that defines $\dens_{\alpha,\eps}^2[\msp](y)$, into (\ref{eq:tEn}), obtaining a triple integral
    \begin{equation*} \label{eq:triple}
        \tEn_{\alpha,\eps}(\msp) =
        \iiint_{M\times M\times M} \varphi_{\alpha,\eps}(d(x,y)) \varphi_{\alpha,\eps}(d(z,y)) \, \dd \msp(x) \dd \msp(z) \, \dd \Leb(y),
    \end{equation*}
    and then change the order of integration.
\end{proof}

The following Proposition formalizes the statement that for singular enough $\msp$ the values $\En_{\alpha, \eps} (\msp)$ and $\tEn_{\alpha, \eps} (\msp)$ are close to each other.
\begin{proposition} \label{prop:EnEquiv}
    For every $\alpha > 0$ and every $\delta > 0$ there exists $C > 0$ such that for every $\eps > 0$ if $\En_{\alpha, \eps} (\msp) > C$ or $\tEn_{\alpha, \eps} (\msp) > C$ then one has
    \begin{equation*}
        \frac{\En_{\alpha, \eps} (\msp)}{\tEn_{\alpha, \eps} (\msp)} \in (1 - \delta, 1 + \delta).
    \end{equation*}
\end{proposition}

In what follows it will be convenient to use the following notation:
\begin{definition}
    We say that two positive numbers, $A$ and $A'$, are \emph{$(\delta,C)$-close}, if
    \begin{equation*}
        A < (1+\delta)A'+C \quad \text{ and} \quad A'< (1+\delta)A+C;
    \end{equation*}
    this can be equivalently rewritten as
    \begin{equation*}
        \frac{1}{1+\delta} A - \frac{1}{1+\delta}C < A'< (1+\delta)A+C.
    \end{equation*}
    We denote it $A \approx_{(\delta,C)} A'$.
\end{definition}

\begin{remark} \label{r:EnEqForm}
    We can reformulate Proposition \ref{prop:EnEquiv} using notations introduced above: for every $\alpha > 0$ and every $\delta > 0$ there exists $C < \infty$, such that for every $\eps > 0$ and every $\msp \in \mM$ one has
    \begin{equation*}
        \En_{\alpha,\eps}(\msp)\approx_{(\delta,C)}\tEn_{\alpha,\eps}(\msp).
    \end{equation*}
\end{remark}

Joining Proposition \ref{prop:EnEquiv} and Corollary \ref{cor:EnHol} we get
\begin{corollary} \label{cor:tEnChangeOneStepHol}
    For every $\delta > 0$ there exists $\eps_0 > 0$, such that for every $R < \infty$ there exists $\alpha_0 > 0$, such that for every $0 < \alpha < \alpha_0$ there exists $C > 0$ such that for every $f \in \Hol(M)$ with $L(f) < R$ and $\gamma(f)^{-1} < R$, every $0 < \eps < \eps_0$ and every $\msp$, such that $\tEn_{\alpha, \eps}(\msp) > C$ or $\En_{\alpha, \eps}(\msp) > C$ one has
    \begin{equation} \label{eq:tEnChangeOneStepHol}
        \frac{\tEn_{\alpha, \eps} (f_*\msp)}{\tEn_{\alpha, \eps} (\msp)} \in (1-\delta, 1+\delta).
    \end{equation}
\end{corollary}

To prove Proposition \ref{prop:EnEquiv} we will need the following:
\begin{lemma} \label{l:K-U-comp}
    For every $\alpha > 0$, the interaction potentials $K_{\alpha,\eps}(x,z)$ and \\$U_{\alpha,\eps}(d(x,z))$ are comparable in the  following sense:
    \begin{multline} \label{eq:K-U-close}
        \forall \delta>0 \, \exists C: \quad \forall \eps>0, \,\forall x,z\in M \\
        K_{\alpha,\eps}(x,z) \approx_{(\delta, C)} U_{\alpha,\eps}(d(x,z)), \ \text{i.e. }\\
        \frac{1}{1+\delta} (U_{\alpha,\eps}(d(x,z))-C) < K_{\alpha,\eps}(x,z) < (1+\delta) U_{\alpha,\eps}(d(x,z))+C.
    \end{multline}
\end{lemma}

Let us postpone the proof of Lemma \ref{l:K-U-comp} for a moment and deduce Proposition~\ref{prop:EnEquiv}.

\begin{proof}[Proof of Proposition~\ref{prop:EnEquiv}]
    It suffices to integrate~\eqref{eq:K-U-close} w.r.t. $\msp\times \msp \, (x,z)$. As constant $C$ does not depend on these points nor on the measure $\msp$, one gets
    \begin{equation*}
        \frac{1}{1+\delta} (\En_{\alpha, \eps}(\msp)-C) < \tEn_{\alpha, \eps}(\msp) < (1+\delta) \En_{\alpha, \eps}(\msp)+C
    \end{equation*}
    with the same constant~$C$. As it was noticed in Remark~\ref{r:EnEqForm}, this is an equivalent form of Proposition~\ref{prop:EnEquiv}.
\end{proof}

\begin{proof}[Proof of Lemma~\ref{l:K-U-comp}]
    Let $\alpha > 0$ be given. Take any $r_0>0$ and divide the integral~\eqref{eq:K-a-e} defining $K_{\alpha,\eps}(x,z)$ into two parts, depending on whether the distance $d(x,y)$ exceeds $r_0$:
    \begin{multline*}
        K_{\alpha,\eps}(x,z) = \int_{B_{r_0}(x)} \varphi_{\alpha,\eps}(d(x,y)) \varphi_{\alpha,\eps}(d(z,y)) \, \dd \Leb(y) + \\
        + \int_{M\setminus B_{r_0}(x)} \varphi_{\alpha,\eps}(d(x,y)) \varphi_{\alpha,\eps}(d(z,y)) \, \dd \Leb(y).
    \end{multline*}
    Denote the first and the second summands as $K_{\alpha,\eps}^{(r_0)}(x,z)$ and  $\bK_{\alpha,\eps}^{(r_0)}(x,z)$. Note that the second summand is uniformly bounded: indeed, the factor $\varphi_{\alpha,\eps}(d(x,y))$ doesn't exceed a constant $\varphi_{\alpha, \eps}(r_0)\le \varphi_{\alpha}(r_0)$, while the second factor $\varphi_{\alpha,\eps}(d(y,z))$ is a function with the integral on~$M$ that is bounded uniformly in~$z\in M$. The latter uniform bound can be seen by again decomposing the integral in two:
    \begin{multline} \label{eq:phi-integral-bound}
        \int_M \varphi_{\alpha,\eps}(d(y,z)) \, \dd \Leb(y) = \int_{B_{r_0}(z)} \varphi_{\alpha,\eps}(d(y,z)) \, \dd \Leb(y) + \\
        + \int_{M\setminus B_{r_0}(z)} \varphi_{\alpha,\eps}(d(y,z)) \, \dd \Leb(y).
    \end{multline}
    The second summand in~\eqref{eq:phi-integral-bound} does not exceed $\vol(M)\cdot \varphi_{\alpha}(r_0)$.The first one can be estimated uniformly in~$z\in M$ by passage to the geodesic coordinates centred at~$z$, comparing it to the same integral in a ball in $\R^k$: the Jacobian of the change of variables is (for $r_0$ smaller than the injectivity radius in $M$) uniformly bounded, and the integral of $\frac{( - \log(r))^{\frac{\alpha - 1}{2}}}{r^{\frac{k}{2}}}$ on $B_{r_0}(0)\subset \R^k$ converges.

    We thus have
    \begin{equation} \label{eq:C-K-alpha}
        \bK_{\alpha,\eps}^{(r_0)}(x,z) \le C^K_{\alpha}(r_0)
    \end{equation}
    for some constant $C^K_{\alpha}(r_0)$.

    Now, let us transform the first integral, $K_{\alpha,\eps}^{(r_0)}(x,z)$. For sufficiently small $r_0>0$ (smaller than the injectivity radius at every point) one can take the geodesic coordinates at any point $x\in M$ in a ball of radius $r_0$. Thus, we can take a point $x'\in \R^k$, let $\psi:B_{r_0}(x')\to B_{r_0}(x)$ be geodesic coordinates, and denote $z':=\psi^{-1}(z)$.

    Fix any $\delta_1>0$. Due to the compactness of $M$, for a sufficiently small $r_0$ the Jacobian and the bi-Lipschitz constants of $\psi$ are $\delta_1$-close to~$1$:
    \begin{equation*}
        \fL(\psi) < 1+\delta_1, \quad \Jac(\psi) \in \left(\frac{1}{1+\delta_1}, 1+\delta_1\right).
    \end{equation*}
    Making a change of variables $y=\psi(y')$ in the integral for~$K_{\alpha,\eps}^{(r_0)}(x,z)$, we get
    \begin{multline*}
        K_{\alpha,\eps}^{(r_0)}(x,z) = \int_{B_{r_0}(x)} \varphi_{\alpha,\eps}(d(x,y)) \varphi_{\alpha,\eps}(d(z,y)) \, \dd \Leb_M(y) = \\
        = \int_{B_{r_0}(x')} \varphi_{\alpha,\eps}(d(\psi(x'),\psi(y'))) \varphi_{\alpha,\eps}(d(\psi(z'),\psi(y'))) \, \dd \Leb_M(\psi(y'));
    \end{multline*}
    all the three quotients
    \begin{equation*}
        \frac{\varphi_{\alpha,\eps}(d(\psi(x'),\psi(y')))}{\varphi_{\alpha,\eps}(|x'-y'|)},
        \quad
        \frac{\varphi_{\alpha,\eps}(d(\psi(z'),\psi(y')))}{\varphi_{\alpha,\eps}(|z'-y'|)},
        \quad
        \left. \frac{\dd\, \Leb_M}{\dd \, \psi_* \Leb_{\R^k}} \right|_{y}
    \end{equation*}
    are close to~$1$: the first two by a factor $(1+\delta_1)^{\frac{k + \alpha}{2}}$ since 
    \begin{equation*}
        1 \le \frac{\varphi_{\alpha, \eps} (r_1)}{\varphi_{\alpha, \eps} (r_2)} \le \left( \frac{\log(r_1)}{\log(r_2)} \right)^{\frac{\alpha - 1}{2}} \left( \frac{r_2}{r_1} \right)^{\frac{k}{2}} < \left( \frac{r_2}{r_1} \right)^{\frac{k + \alpha}{2}}  \quad \text{for any $0 < r_1 \le r_2 < \frac{1}{e}$,}
    \end{equation*}
    and the last one by the factor $(1+\delta_1)$. Thus, $K_{\alpha,\eps}^{(r_0)}(x,z)$ differs from
    \begin{equation} \label{eq:I-r_0}
        I_{\alpha,\eps}^{(r_0)} (x',z') :=\int_{B_{r_0}(x')} \varphi_{\alpha,\eps}(|x'-y'|) \varphi_{\alpha,\eps}(|z'-y'|) \, \dd \Leb_{\R^k}(y')
    \end{equation}
    by the factor at most $(1+\delta_1)^{k+\alpha+1}$. Now, $I_{\alpha,\eps}^{(r_0)} (x',z')$ is also a part of $I_{\alpha,\eps} (x',z')$ (define by (\ref{eq:I})), that differs from it by
    \begin{equation*}
        \overline{I}_{\alpha,\eps}^{(r_0)} (x',z') :=\int_{\R^k\setminus B_{r_0}(x')} \varphi_{\alpha,\eps}(|x'-y'|) \varphi_{\alpha,\eps}(|z'-y'|) \, \dd \Leb_{\R^k}(y'),
    \end{equation*}
    that is bounded uniformly in $\eps, x', z'$ (for $|x'-z'|<\frac{r_0}{2}$ both functions outside of $B_{r_0} (x')$ are bounded and have compact supports, for $|x'-z'|\ge \frac{r_0}{2}$  we have $\overline{I}_{\alpha,\eps}^{(r_0)} (x',z') \le U_{\alpha, \eps} \left( \frac{r_0}{2} \right)$). Thus, first choosing $\delta_1$ so that $(1+\delta_1)^{k + \alpha + 1}<1+\delta$ and then accordingly choosing $r_0$, we obtain the desired estimate~\eqref{eq:K-U-close}.

\end{proof}

\subsection{Properties of $\theta_{\alpha, \eps} (\msp)$ and $\tEn_{\alpha, \eps} (\msp)$ under H\"older homeomorphisms} \label{subsection:theta}

First, in this section we show that computing the measure $\theta_{\alpha, \eps}$ ``almost commutes'' with a H\"older continuous homeomorphisms, i.e. $\theta_{\alpha, \eps} (f_* \msp) \simeq f_* \theta_{\alpha, \eps} (\msp)$ (see Proposition \ref{prop:WassEst}). Second, we show that the energy $\tEn_{\alpha, \eps} (\msp)$ doesn't change much on average if we apply a random H\"older homeomorphism (see Proposition \ref{lm:tailsEst}). Both facts are used to prove the contraction described in part \ref{iii} of the outline of the proof.

In what follows it will be convenient to use Wasserstein metric in the space of probability measures on~$M$. Let us recall its definition, as well as definition of the total variation distance between measures:

\begin{definition}
    Let $\msp_1, \msp_2$ be two probability measures on a measure space $(M, \mathcal{B})$. Then the \emph{Wasserstein distance} between them is defined as
    \begin{equation*}
        W(\msp_1, \msp_2) = \inf_{\gamma} \iint_{M\times M} d(x, y) \, \dd \gamma(x, y),
    \end{equation*}
    where the infimum is taken over all probability measures $\gamma$ on $(M \times M, \mathcal{B} \times \mathcal{B})$ with the marginals (projections on the $x$ and $y$ coordinates) $P_x(\gamma) = \msp_1$ and $P_y(\gamma) = \msp_2$.
\end{definition}
\begin{definition}
    Let $\msp_1, \msp_2$ be two probability measures on a measure space $(M, \mathcal{B})$. Then the \emph{total variation} distance between them is
    \begin{equation*}
        \TV(\msp_1, \msp_2) = \sup_{B \in \mathcal{B}} |\msp_1(B) - \msp_2(B)|.
    \end{equation*}
\end{definition}

Also, we will need the following statement that can be found, for example, in \cite[Theorem 6.15]{Vi}: the Wasserstein metric is bounded from above by the diameter times the total variation distance.
\begin{lemma} \label{lm:WasEst}
    For every two probability measures $\msp_1, \msp_2 $ on a manifold $M$ we have
    \begin{equation*}
        W(\msp_1, \msp_2) \le \diam(M) \cdot \TV(\msp_1, \msp_2).
    \end{equation*}
\end{lemma}

Equipped with these notions we would like to prove the following
\begin{proposition} \label{prop:WassEst}
    For every $\delta > 0$ and every $R < \infty$ there exists $\alpha_0 > 0$, such that for every $0 < \alpha < \alpha_0$ there exists $C > 0$, such that for every $\eps > 0$, every $f \in \Hol(M)$, such that $\gamma(f)^{-1}, L(f) < R$ and every $\msp \in \mM$ such that $\En_{\alpha, \eps} (\msp) > C$ or $\tEn_{\alpha, \eps} (\msp) > C$ one has
    \begin{equation} \label{tEnWassInv}
        W(f_* \theta_{\alpha, \eps} [\msp], \theta_{\alpha, \eps} [f_* \msp]) < \delta.
    \end{equation}
\end{proposition}

\begin{proof} 
    For given $\alpha,\eps>0$ and measure $\msp$, consider a (non-probability) measure $\dm_{\alpha,\eps}(\msp)$ on $M\times M \times M$, given by
    \begin{equation*}
        \dm_{\alpha,\eps}(\msp)= \varphi_{\alpha,\eps}(d(x,y)) \varphi_{\alpha,\eps}(d(z,y)) \, \dd \msp(x) \, \dd \Leb(y) \, \dd \msp(z).
    \end{equation*}

    Denote by $\pi_1$, $\pi_2$ the projections of $M\times M\times M$ on the first and second coordinates respectively, and by $\pi_{1,3}$ the projection on $M\times M$ corresponding to the first and third coordinates. Then directly from definition
    \begin{equation*}
        (\pi_{2})_* \, \dm_{\alpha,\eps}(\msp) = \mT_{\alpha,\eps}(\msp);
    \end{equation*}
    also, define
    \begin{equation} \label{eq:pi-1-3}
        \mTt_{\alpha,\eps}(\msp):= (\pi_{1,3})_* \, \dm_{\alpha,\eps}(\msp) = K_{\alpha,\eps}(x,z) \, \msp(dx) \, \msp(dz);
    \end{equation}
    the second equality is due to~\eqref{eq:K-a-e}.
    Finally, consider the measure
    \begin{equation}
        \mTf_{\alpha,\eps}(\msp):=(\pi_1)_* \, \dm_{\alpha,\eps}(\msp)
    \end{equation}
    as well as the normalizations of these measures,
    \begin{equation}
        \mTtn_{\alpha,\eps}(\msp):=\frac{1}{\tEn_{\alpha,\eps}(\msp)}  \mTt_{\alpha,\eps}(\msp), \quad
        \mTfn_{\alpha,\eps}(\msp):=\frac{1}{\tEn_{\alpha,\eps}(\msp)} \mTf_{\alpha,\eps}(\msp).
    \end{equation}

    For a high-energy measure $\msp$ most of the measure $\dm_{\alpha,\eps}(\msp)$ is concentrated near the diagonal, and hence the projections on the first and on the second coordinates are close to each other. The following lemma formalizes this argument:

    \begin{lemma} \label{l:W-close}
        For every $\delta_1$, $\alpha$ there exists $C'$ such that for every $\eps>0$ and $\msp$ with $\tEn_{\alpha,\eps}(\msp)>C'$ one has
        \begin{equation*}
            W(\mTn_{\alpha,\eps}(\msp),  \mTfn_{\alpha,\eps}(\msp))<\delta_1.
        \end{equation*}
    \end{lemma}

    \begin{proof}
        Take $r_0:=\frac{\delta_1}{2}$ and let
        \begin{equation*}
            A_{r_0}:=\{(x,y,z)\in M^3 \mid d(x,y)< r_0 \}, \quad \bA_{r_0}:=M^3 \setminus A_{r_0}.
        \end{equation*}
    
        From the proof of Lemma~\ref{l:K-U-comp}, we have
        \begin{equation*}
            \dm_{\alpha,\eps}(\msp)(\bA_{r_0}) = \int \bK_{\alpha,\eps}^{(r_0)}(x,z) \, d\msp(x) \, d\msp(z) \le C^K_{\alpha}(r_0),
        \end{equation*}
        where the second inequality is due to~\eqref{eq:C-K-alpha}. Thus, the non-normalized measure $\dm_{\alpha,\eps}(\msp)(\bA_{r_0})$ does not exceed a constant~$C^K_{\alpha}(r_0)$.

        Now, we can couple the normalized measures $\mTn_{\alpha,\eps}(\msp)$ and $\mTfn_{\alpha,\eps}(\msp)$ using the projection of $\frac{1}{\tEn_{\alpha,\eps}(\msp)} \dm_{\alpha,\eps}(\msp)$ on the first two coordinates; this coupling leads to the upper bound for the Wasserstein distance:
        \begin{multline}\label{eq:W-delta-1}
            W(\mTn_{\alpha,\eps}(\msp), \mTfn_{\alpha,\eps}(\msp))\le r_0 \cdot \frac{\dm_{\alpha,\eps}(\msp)(A_{r_0})}{\tEn_{\alpha,\eps}(\msp)} \\
            + \diam(M) \cdot \frac{\dm_{\alpha,\eps}(\msp)(\bA_{r_0})}{\tEn_{\alpha,\eps}(\msp)}    \le r_0+ \frac{C^K_{\alpha}(r_0)}{\tEn_{\alpha,\eps}(\msp)},
        \end{multline}
        where the first summand corresponds to the points with $d(x,y)<r_0$, and the second to the points with $d(x,y)\ge r_0$. As we chose $r_0=\frac{\delta_1}{2}$, it suffices to require that
        \begin{equation*}
            \tEn_{\alpha,\eps}(\msp) > \frac{2 \, C^K_{\alpha} (r_0)}{\delta_1} =: C'
        \end{equation*}
        to ensure that the total Wasserstein distance does not exceed $\delta_1$.
    \end{proof}

    On the other hand, for the measures $\mTfn_{\alpha,\eps}(\msp)$ the analogue of Proposition~\ref{prop:WassEst} can be established directly, and one can even estimate the total variations distance:
    \begin{lemma}\label{l:TV-prime}
        For every $\delta_2 > 0$ and every $R > 0$ there exists $\alpha_2 > 0$ such that for every $0 < \alpha < \alpha_2$ there exists $C > 0$ such that for every $\eps > 0$, every $f \in \Hol(M)$ with $\gamma(f)^{-1}, L(f) < R$ and every $\msp$ such that $\tEn_{\alpha, \eps} (\msp) > C$ or $\En_{\alpha, \eps} (\msp) > C$  one has
        \begin{equation}\label{eq:TV-theta-prime}
            \TV(f_* \mTfn_{\alpha,\eps}(\msp),\mTfn_{\alpha,\eps}(f_* \msp)) < \delta_2,
        \end{equation}
        and, hence,
        \begin{equation}\label{eq:W-TV}
            W(f_* \mTfn_{\alpha,\eps}(\msp),\mTfn_{\alpha,\eps}(f_* \msp)) < \delta_2 \cdot \diam(M).
        \end{equation}
    \end{lemma}

    \begin{proof}

        We will first use Lemma~\ref{l:K-U-comp} together with Part \textbf{(IV)} of  Proposition~\ref{prop:UAEps}
        to compare $K_{\alpha,\eps}(x,z)$ with $K_{\alpha,\eps}(f(x),f(z))$. Fix $\alpha_0$ and $\delta_3$ so small that
        \begin{equation*}
            (1+\delta_3)^2 \cdot \left[ 2 R (1 + \log(R)) \right]^{\alpha_0} < 1+\frac{\delta_2}{10}.
        \end{equation*}
        Then from Lemma~\ref{l:K-U-comp} we know that there exists $C>0$ such that for any $x,z,\eps$ we have
        \begin{equation} \label{KUequiv}
            K_{\alpha,\eps}(x,z) \approx_{(\delta_3,C)} U_{\alpha,\eps}(x,z), \quad  K_{\alpha,\eps}(f(x),f(z)) \approx_{(\delta_3,C)} U_{\alpha,\eps}(f(x),f(z)).
        \end{equation}
        Without loss of generality we can assume that $d(x, z) < d(f(x), f(z))$. Then by Part \textbf{(IV)} of Proposition \ref{prop:UAEps} we can choose a particular $r_0 = r_0(\alpha_0, \delta_3)$, such that if $d(f(x), f(z)) < r_0$ then  
        \begin{equation} \label{ineq:delta3}
            \frac{U_{\alpha, \eps} (d(x, z))}{U_{\alpha, \eps} (d(f(x), f(z)))} < (1 + \delta_3) \left( \frac{\log(d(x, z))}{\log(d(f(x), f(z)))} \right)^{\alpha}.
        \end{equation}
        Notice that inequality (\ref{ineq:delta3}) only holds for $\eps < \eps_0(\alpha_0, \delta_3)$. That can be guaranteed by taking $C$ big enough, because, according to Lemma \ref{lm:EnUpperBound}, the energy $\En_{\eps, \alpha} (\msp)$ for $\eps \ge \eps_0$ has an explicit upper bound.
        
        Using the fact that $L(f), \gamma(f)^{-1} < R$ we conclude that
        \begin{multline*}
            \left( \frac{\log(d(x, z))}{\log(d(f(x), f(z)))} \right)^{\alpha} \le \\
            \le \left( \frac{\frac{2}{\gamma(f)} \log \left[ d(f(x), f(z)) L(f)^{-1} \right]}{\log(d(f(x), f(z)))} \right)^{\alpha} < \\
            < \left[ 2 R (1 + \log(R)) \right]^{\alpha_0}
        \end{multline*}
        and
        \begin{equation} \label{delta2Eq}
            U_{\alpha, \eps} (d(x, z)) < (1 + \delta_3) \left[ 2 R (1 + \log(R)) \right]^{\alpha_0} U_{\alpha, \eps} (d(f(x), f(z))).
        \end{equation}
        On the other hand, if $d(f(x), f(z)) > r_0$ then
        \begin{equation*}
            d(x, z) \ge \left( \frac{r_0}{L(f)} \right)^{\frac{2}{\gamma(f)}} > \left( \frac{r_0}{R} \right)^{\frac{2}{R}}
        \end{equation*}
        and hence 
        \begin{equation} \label{C''Eq}
            U_{\alpha, \eps} (d(x, z)) < U_{\alpha_0} \left( \left( \frac{r_0}{R} \right)^{\frac{2}{R}} \right).
        \end{equation}
        Combining estimates (\ref{KUequiv}), (\ref{delta2Eq}) and (\ref{C''Eq}) we conclude that
        \begin{equation} \label{eq:K-delta-C}
            K_{\alpha,\eps}(x,z) \approx_{(\frac{\delta_2}{10},C'')}K_{\alpha,\eps}(f(x),f(z))
        \end{equation}
        for some explicit constant~$C''$.

        Now, applying $f^{-1}_*$ does not change the total variations distance, so instead of~\eqref{eq:TV-theta-prime} we can show the equivalent statement
        \begin{equation}\label{eq:d-2}
            \TV(\mTfn_{\alpha,\eps}(\msp),f^{-1}_* \mTfn_{\alpha,\eps}(f_* \msp)) < \delta_2.
        \end{equation}
        The measures here can be obtained as projections of measures on $M\times M$:
        \[
            \mTfn_{\alpha,\eps}(\msp) =  (\pi_1)_* \, \mTtn_{\alpha,\eps}(\msp),
        \]
        \[
            f^{-1}_* \mTfn_{\alpha,\eps}(f_* \msp) =
            (\pi_1)_* \left(f^{-1}_* \mTtn_{\alpha,\eps}(f_* \msp)\right)
        \]
        where $\pi_1$ is the projection on the first coordinate. To obtain the desired~\eqref{eq:d-2}, we will actually show that for sufficiently high energy~$\tEn_{\alpha,\eps}(\msp)$ even before projection one has
        \begin{equation}\label{eq:TV-delta-2}
            \TV(\mTtn_{\alpha,\eps}(\msp),f^{-1}_* \mTtn_{\alpha,\eps}(f_* \msp))<\delta_2.
        \end{equation}

        To do so, note that both these measures are absolutely continuous with respect to $\nu\times \nu$:
        \[
            \mTtn_{\alpha,\eps}(\msp) = \frac{1}{\tEn_{\alpha,\eps}(\msp)}  K_{\alpha,\eps}(x,z) \, \nu(dx) \, \nu (dz),
        \]
        \[
            f^{-1}_* \mTtn_{\alpha,\eps}(f_* \msp) =  \frac{1}{\tEn_{\alpha,\eps}(f_* \msp)} K_{\alpha,\eps}(f(x),f(z)) \, \nu(dx) \, \nu (dz).
        \]
        Now, \eqref{eq:K-delta-C} implies that once $K_{\alpha,\eps}(x,z)>C_2 := \frac{20C''}{\delta_2}$, the $K_{\alpha,\eps}$ parts of these densities are close to each other:
        \[
            \frac{1}{1+\frac{\delta_2}{6}}< \frac{K_{\alpha,\eps}(f(x),f(z))}{K_{\alpha,\eps}(x,z)} < 1+\frac{\delta_2}{6}.
        \]
        On the other hand, due to Corollary~\ref{cor:tEnChangeOneStepHol}, once the energy $\tEn_{\alpha,\eps}(\msp)$ is sufficiently large, the normalization constants are also close to each other:
        \[
            \frac{1}{1+\frac{\delta_2}{6}} \, \tEn_{\alpha,\eps}(\msp)<\tEn_{\alpha,\eps}(f_*\msp) <
            (1+\frac{\delta_2}{6}) \, \tEn_{\alpha,\eps}(\msp).
        \]
        Multiplying the two inequalities, we get that on the set $$A:=\{(x,z)\mid K_{\alpha,\eps}(x,z)>C_2\}$$ the quotient of densities w.r.t. $\msp\times \msp$ of the two measures is in the interval 
        \[
            \left( \frac{1}{ \left( 1+\frac{\delta_2}{6} \right)^2}, \left( 1+\frac{\delta_2}{6} \right)^2\right).
        \]  

        Finally, the $\mTtn_{\alpha,\eps}(\msp)$-measure of its complement does not exceed $\frac{C_2}{\tEn_{\alpha,\eps}(\msp)}$. Hence, if the energy $\tEn_{\alpha,\eps}(\msp)$ is sufficiently high to make sure that $\frac{C_2}{\tEn_{\alpha,\eps}(\msp)}<\frac{\delta_2}{6}$, the part that the normalized measures have in common is at least
        \[
            \frac{1-\frac{\delta_2}{6}}{\left( 1 + \frac{\delta_2}{6} \right)^2} > 1-\frac{\delta_2}{2},
        \]
        and hence the total variation~\eqref{eq:TV-delta-2} indeed does not exceed~$\delta_2$.

        An application of Lemma~\ref{lm:WasEst} concludes the proof of the upper bound~\eqref{eq:W-TV} for the Wasserstein distance.

    \end{proof}

    Lemma \ref{l:W-close}  and   Lemma \ref{l:TV-prime}  
    together imply Proposition~\ref{prop:WassEst}. Indeed, we have:
    \begin{multline*}
        W(f_* \mTn_{\alpha,\eps}(\msp), \mTn_{\alpha,\eps}(f_* \msp) )
        \le
        W(f_* \mTn_{\alpha,\eps}(\msp), f_* \mTfn_{\alpha,\eps}(\msp))+
        \\
        + W(f_* \mTfn_{\alpha,\eps}(\msp),\mTfn_{\alpha,\eps}(f_* \msp))
        + W(\mTfn_{\alpha,\eps}(f_* \msp), \mTn_{\alpha,\eps}(f_* \msp))
    \end{multline*}
    The first and the third summands can be estimated directly using Lemma~\ref{l:W-close}. It suffices to note that since $L(f), \gamma(f)^{-1} < R$ we have control over the increase of the Wasserstein distance after application of~$f$. Indeed, a simple application of Jensen's inequality gives us 
    \begin{equation*}
        W(f_* \mTn_{\alpha,\eps}(\msp), f_* \mTfn_{\alpha,\eps}(\msp)) < R \cdot W(\mTn_{\alpha,\eps}(\msp), \mTfn_{\alpha,\eps}(\msp))^{\frac{2}{R}}.
    \end{equation*}
    For the last summand, to apply Lemma, we note that an upper bound on $L(f)$ and $\gamma(f)^{-1}$ implies that $f_*\msp$ is of sufficiently high energy provided that $\msp$ is of high energy.

    Finally, the second summand is estimated directly by Lemma~\ref{l:TV-prime}.
\end{proof}

\begin{proposition} \label{lm:tailsEst}
    If a probability measure $\mgr$ on $\Hol(M)$ satisfies assumption (\ref{RegCondHol}) then for every $\delta > 0$ there exists $\alpha_0 > 0$ such that for every $0 < \alpha < \alpha_0$ there exists $C > 0$ such that for every $\eps > 0$ and every measure $\msp \in \mM$ such that $\En_{\alpha, \eps} (\msp) > C$ the following formula holds:
    \begin{equation} \label{tailsEst1}
        \frac{\E_{\mgr} \left[ \tEn_{\alpha, \eps} (f_* \msp) \right]}{\tEn_{\alpha, \eps} (\msp)} \in (1 - \delta, 1 + \delta).
    \end{equation}
\end{proposition}

\begin{proof}

    Let us fix some $\dn > 0$ and pick $\tdelta > 0$, such that $(1 + \tdelta)^2 < 1 + \dn$. It follows directly from condition (\ref{RegCondHol}) and H\"older inequality that there exists $\alpha_0 > 0$, such that for every $0 < \alpha < \alpha_0$ one has
    \begin{equation} \label{IntEst}
        \int_{\Hol(M)} \Lambda(f)^{\alpha} \; \dd \mgr(f) < 1 + \tdelta.
    \end{equation}
    According to Corollary \ref{cor:EnHol} we can choose $C$ such that inequality (\ref{EnHolOneStep}) holds for $\tdelta$. Let $f \in \Hol(M)$ be a random homeomorphism distributed with respect to the measure $\mgr$. Then integrating inequality (\ref{EnHolOneStep}) with respect to the measure $\mgr$ and combining it with inequality (\ref{IntEst}) we get:
    \begin{equation} \label{eq:E-f-power}
        \frac{1}{1 + \dn} < \frac{1}{(1 + \tdelta)^2} \le \frac{\E_{\mgr} \left[ \En_{\alpha, \eps} (f_* \msp) \right]}{\En_{\alpha, \eps} (\msp)} \le (1 + \tdelta)^2 < 1 + \dn.
    \end{equation}

    Let us now return back to the original~\eqref{tailsEst1}. 
    Namely, let $\delta > 0$ be given; choose and fix $\dn > 0$ such that $(1 + \dn)^3 < 1 + \delta$, and let $\alpha_0$ be chosen w.r.t. $\dn$ so that~\eqref{eq:E-f-power} holds.

    By Proposition~\ref{prop:EnEquiv},  
    there exists $\Cn>0$ such that for any measure $\msp'$ on $M$ and any $\eps>0$ one has
    \begin{equation}\label{eq:nu-arb}
        \frac{1}{1+\dn} \, \En_{\alpha,\eps}(\msp') - \frac{\Cn}{1+\dn}  <
        \tEn_{\alpha,\eps}(\msp')  < (1+\dn) \En_{\alpha,\eps}(\msp') + \Cn
    \end{equation}
    Applying this for $\msp'=f_*\msp$ and taking the expectation w.r.t. $\mgr$ provides
    \begin{multline*}
        \frac{1}{1+\dn} \, \E_{\mu}\left [\En_{\alpha,\eps}(f_*\msp) \right] - \frac{\Cn}{1+\dn}  <
        \E_{\mu}\left [\tEn_{\alpha,\eps}(f_*\msp) \right]
        \\
        < (1+\dn) \E_{\mu}\left [\En_{\alpha,\eps}(f_*\msp) \right] + \Cn;
    \end{multline*}
    using~\eqref{eq:E-f-power}, we thus get
    \begin{equation*}
        \frac{1}{(1+\dn)^2} \, \En_{\alpha,\eps}(\msp) - \frac{\Cn}{1+\dn}  <
        \E_{\mu}\left [\tEn_{\alpha,\eps}(f_*\msp) \right]
        \\
        < (1+\dn)^2 \En_{\alpha,\eps}(\msp)  + \Cn.
    \end{equation*}
    Finally, using~\eqref{eq:nu-arb} with $\msp'=\msp$ to estimate $\En_{\alpha,\eps}$ via $\tEn_{\alpha,\eps}$, we get
    \begin{multline*}
        \frac{1}{(1+\dn)^2} \cdot \frac{1}{(1+\dn)}\left(\tEn_{\alpha,\eps}(\msp) - \Cn\right) - \frac{\Cn}{1+\dn}  <
        \E_{\mu}\left [\tEn_{\alpha,\eps}(f_*\msp) \right]
        \\
        < (1+\dn)^2 \cdot \left ((1+\dn) \tEn_{\alpha,\eps}(\msp) + \Cn\right)  + \Cn;
    \end{multline*}
    as $(1+\dn)^3<1+\delta$, we obtain the desired
    \[
        (1+\delta) (\tEn_{\alpha,\eps}(\msp) - C') < \E_{\mu}\left [\tEn_{\alpha,\eps}(f_*\msp) \right] <
        (1+\delta) \tEn_{\alpha,\eps}(\msp) + C'
    \]
    for some constant~$C'$. As $\delta>0$ was arbitrary \eqref{tailsEst1} follows. 

\end{proof}

\begin{remark}
    Notice that if instead of one measure $\mgr$ on $\Hol(M)$ we consider a compact $K$ of such measures, such that there exists $\beta > 0$ and $C_0$ such that for every $\mgr \in K$
        \begin{equation*}
            \int_{\Hol(M)} \Lambda(f)^{\beta} \; \dd \mgr(f) < C_0
        \end{equation*}
    then constants $\alpha_0$ and $C$ in Proposition \ref{lm:tailsEst} can be chosen uniformly for all measures in $K$.
\end{remark}

\subsection{Contraction of $\tEn_{\alpha, \eps}$ under random dynamics} \label{subsection:contr}

In this section we prove a proposition that formalizes \eqref{contrInt}. Namely, we show that if $\En_{\alpha, \eps} (\msp)$ is big enough then $\tEn_{\alpha, \eps} (\mgr * \msp) < \lambda \tEn_{\alpha, \eps} (\msp)$ for some $\lambda < 1$. This will allow us to control the energy after $n$ iterations of random dynamics and finish the proof of Theorem \ref{thm:mainHol-3}.

\begin{proposition} \label{prop:contr}
    Under the assumptions of Theorem \ref{thm:mainHol-3} there exists $\alpha > 0$, $\lambda < 1$, and $C < \infty$, such that for every $\eps > 0$, every $\msp \in \mM$ and every $\mgr \in K$ if $\En_{\alpha, \eps} (\msp) > C$ then
    \begin{equation} \label{EnContr}
        \tEn_{\alpha, \eps} (\mgr * \msp) < \lambda \tEn_{\alpha, \eps} (\msp).
    \end{equation}
\end{proposition}

\begin{proof}

    Fix some small number $\delta > 0$ , take $\lambda = 1 - \delta$, and assume the contrary: for any $\alpha > 0$ and $C < \infty$ there exists a measure $\mgr \in K$, an $\eps > 0$ and a probability measure $\msp$ with energy $\tEn_{\alpha,\eps}(\msp) > C$, such that
	\begin{equation} \label{contrAssumption2}
		\tEn_{\alpha, \eps} (\mgr * \msp) \ge (1-\delta) \tEn_{\alpha, \eps} (\msp).
	\end{equation}

    Recall the following standard statement:
    \begin{lemma} \label{l:var}
        Assume that in some Hilbert space $\mathcal{H}$ a probability measure is given; in other words, one is given a random variable $v$ taking values in this space. Assume that this measure has finite second moment, and let $\bv:=\E v$ be its expectation. Then
        \begin{equation} \label{eq:v}
            \E \langle v-\bv,v-\bv\rangle = \E \langle v,v\rangle - \langle \bv,\bv\rangle
        \end{equation}
    \end{lemma}

    Let us apply this statement to~$L^2 (M, \text{\rm Leb})$. Namely, for a given measure $\msp$ its averaged image $\mgr*\msp$ is the expectation of $f_*\msp$, where the homeomorphism $f$ of $M$ is taken randomly w.r.t. the measure~$\mgr$.
    Hence, the same applies for the density~$\dens_{\alpha,\eps}[\msp]$, given by~\eqref{eq:dens-a-e} (see Definition~\ref{def:dens}):
    \begin{equation}\label{eq:rho-expect}
        \dens_{\alpha,\eps}[\mgr*\msp] = \E_{\mgr}  \, \dens_{\alpha,\eps}[f_* \msp].
    \end{equation}
    Substituting this into~\eqref{eq:v} (with $v=\dens_{\alpha,\eps}[f_* \msp]$), and taking into account the definition~\eqref{eq:tEn}, we get
    \begin{multline} \label{ContrThmVarFormula}
        \mathbb{E}_{\mgr} \left[ \int_M \left( \dens_{\alpha, \eps}[f_*\msp] (y) - \dens_{\alpha, \eps}[\mgr * \msp] (y) \right)^2 \dd \Leb(y) \right]=
	\\
        = \mathbb{E}_{\mgr} \left[ \tEn_{\alpha, \eps} (f_* \msp) \right] - \tEn_{\alpha, \eps} (\mgr * \msp);
    \end{multline}
    in particular,
    \begin{equation}\label{eq:avg-bound}
        \tEn_{\alpha, \eps} (\mgr * \msp) \le \mathbb{E}_{\mgr} \left[ \tEn_{\alpha, \eps} (f_* \msp) \right].
    \end{equation}

    By Proposition~\ref{lm:tailsEst} we can find sufficiently small $\alpha  > 0$ and sufficiently large $C < \infty$ (uniformly in $\mgr$), such that for any $\eps>0$ and any measure $\msp$ on $M$ with $\tEn_{\alpha, \eps}(\msp) > C$ the inequality~\eqref{tailsEst1} holds, and thus (taking only the upper bound)
    \begin{equation}\label{eq:E-f-delta}
        \mathbb{E}_{\mgr} \left[ \tEn_{\alpha, \eps} (f_* \msp) \right]
        \le (1 + \delta)\tEn_{\alpha, \eps} (\msp).
    \end{equation}

    Denote by $\msp$ a probability measure with $\tEn_{\alpha, \eps}(\msp) > C$, for which the inequality~\eqref{contrAssumption2} holds. Substituting~\eqref{eq:E-f-delta} and~\eqref{contrAssumption2} in the right hand side of~\eqref{ContrThmVarFormula}, we get

    \begin{equation} \label{eq:exp-rho}
        \mathbb{E}_{\mgr} \left[ \int_M \left( \dens_{\alpha, \eps} (f_* \msp) (y) - \dens_{\alpha, \eps} (\mgr * \msp) (y) \right)^2 d \Leb(y) \right] < 2 \delta \, \tEn_{\alpha, \eps} (\msp).
    \end{equation}
    and thus
    \begin{equation} \label{ExpEst1}
        \mathbb{E}_{\mgr} \left[ \int_M \frac{\left( \dens_{\alpha, \eps} (f_* \msp) (y) - \dens_{\alpha, \eps} (\mgr * \msp) (y) \right)^2 }{\tEn_{\alpha, \eps} (\msp)} d \Leb(y) \right] < 2 \delta .
    \end{equation}

    The final step is to show that measures
    $$
        \mTn_{\alpha, \eps} (f_* \msp)=\dfrac{\dens_{\alpha, \eps} (f_* \msp) (y)^2}{\tEn_{\alpha, \eps} (f_* \msp)} d \Leb(y) \ \  \text{and} \ \ \mTn_{\alpha, \eps}(\mgr * \msp)=\dfrac{\dens_{\alpha, \eps} (\mgr * \msp) (y)^2}{\tEn_{\alpha, \eps} (\mgr * \msp)} d \Leb(y)
    $$
    are close with high probability. And indeed, we have the following statement, estimating (even) the total variations distance:
    \begin{lemma} \label{l.West}
        Under the assumptions above
        \begin{equation}\label{eq:Probably-TV}
            \p_{\mgr} \left[ \TV (\mTn_{\alpha, \eps} (f_* \msp), \mTn_{\alpha, \eps}(\mgr * \msp)) > 10 \sqrt[8]{\delta} \right] < 4 \sqrt{\delta},
        \end{equation}
        and hence
        \begin{equation}\label{eq:Probably}
            \p_{\mgr} \left[ W(\mTn_{\alpha, \eps} (f_* \msp), \mTn_{\alpha, \eps}(\mgr * \msp)) > 10 \diam(M) \sqrt[8]{\delta} \right] < 4 \sqrt{\delta}.
        \end{equation}
    \end{lemma}
    \begin{proof}[Proof of Lemma \ref{l.West}]
        We start by working with the non-normalized measures  $\mT_{\alpha,\eps}(f_*\msp)$ and $\mT_{\alpha,\eps}(\mgr*\msp)$. To simplify the notation, denote
        \[
            g_f(y):=\rho_{\alpha,\eps}[f_*\msp](y), \quad \bg(y):=\rho_{\alpha,\eps}[\mgr*\msp](y).
        \]
        Therefore, we have
        $$
            \mT_{\alpha,\eps}[f_*\msp]=(g_f(y))^2d\text{Leb}(y)\ \ \text{\rm and}\ \ \mT_{\alpha,\eps}[\mgr*\msp]=(\bg(y))^2d\text{Leb}(y).
        $$
        From~\eqref{contrAssumption2} we have
        \[
            \tEn_{\alpha, \eps} (\msp) \le \frac{1}{1-\delta} \tEn_{\alpha, \eps} (\mgr * \msp).
        \]
        Joining with~\eqref{eq:exp-rho}, we have
        \begin{equation} \label{eq:exp-rho-f}
            \E_f \left[ \int_M |g_f (y) - \bg (y) |^2 \, d\Leb(y)\right] \le \frac{2\delta}{1-\delta} \tEn_{\alpha, \eps} (\mgr * \msp).
        \end{equation}
        By Markov inequality, with the probability at least $1-\frac{2\sqrt{\delta}}{1-\delta}$, one has
        \begin{multline} \label{eq:f-bar}
            \int_M |g_f (y) - \bg (y) |^2 \, d\Leb(y) \le \sqrt{\delta}\,  \tEn_{\alpha, \eps} (\mgr * \msp)
            \\
            = \sqrt{\delta} \, \mT_{\alpha, \eps} [\mgr* \msp] (M)
            = \sqrt{\delta} \int_M \bg^2(y) \, d\Leb(y).
        \end{multline}
        Now, for any $f$ for which~\eqref{eq:f-bar} holds, consider the set
        \[
            X_f:=\{y\in M \, : \, |g_f (y) - \bg (y) | \le \sqrt[8]{\delta} \cdot \bg(y)\}.
        \]
        Again from Markov inequality type argument, one has
        \begin{equation} \label{eq:M-Xf}
            \mT_{\alpha, \eps} [\mgr* \msp] (X_f) \ge \left( 1-\sqrt[4]{\delta}\right) \mT_{\alpha, \eps} [\mgr* \msp] (M).
        \end{equation}
        Indeed, on the complement $M\setminus X_f$ one has $|g_f (y) - \bg (y) |^2> \sqrt[4]{\delta} \cdot \bg^2(y)$; integrating, one gets
        \[
            \int_{M\setminus X_f} |g_f (y) - \bg (y) |^2 \, d\Leb(y) > \sqrt[4]{\delta} \cdot \mT_{\alpha, \eps} [\mgr* \msp] (M\setminus X_f).
        \]
        Thus, if~\eqref{eq:M-Xf} did not hold, it would imply
        \[
            \mT_{\alpha, \eps} [\mgr* \msp] (M\setminus X_f) > \sqrt[4]{\delta} \cdot  \mT_{\alpha, \eps} [\mgr* \msp] (M),
        \]
        hence providing a contradiction with~\eqref{eq:f-bar}.

        Now,~\eqref{eq:M-Xf} implies that the non-normalized measures $\mT_{\alpha,\eps}[f_*\msp]$ and $\mT_{\alpha,\eps}[\mgr*\msp]$ share a common part
        \[
            \Xi_f:=(1-\sqrt[8]{\delta})^2 \, \bg^2(y) \cdot \Ind_{X_f}(y) \, d\Leb(y)
        \]
        of measure at least
        \begin{equation} \label{eq:Xi}
            \Xi_f(M)\ge (1-\sqrt[8]{\delta})^2 (1-\sqrt[4]{\delta}) \cdot \mT_{\alpha, \eps} [\mgr* \msp] (M).
        \end{equation}

        Finally, note that the normalization constants are also close with high probability, namely, for $f$ such that the inequality~\eqref{eq:f-bar} holds. Indeed,  the inequality~\eqref{eq:f-bar}  can be rewritten as
        \[
            \|g_f-\bg \|_{L_2(M)}^2 \le \sqrt{\delta} \cdot \|\bg \|_{L_2(M)}^2;
        \]
        hence, for any $f$ for which it holds, the triangle inequality implies
        \begin{equation} \label{eq:mTf}
            \mT_{\alpha, \eps} [f_* \msp] (M) =
            \|g_f\|_{L_2(M)}^2 \le (1+\sqrt[4]{\delta})^2
            \mT_{\alpha, \eps} [\mgr* \msp] (M).
        \end{equation}
        Now, the normalized measures $\mTn_{\alpha, \eps} [f_* \msp]$ and $\mTn_{\alpha, \eps} [\mgr * \msp]$ share the common part $\frac{1}{C_f} \Xi_f$, where
        \[
            C_f=\max(\mT_{\alpha, \eps} [f_* \msp](M), \mT_{\alpha, \eps} [\mgr * \msp](M))
        \]
        is the maximum of two normalization constants. Using~\eqref{eq:Xi} and~\eqref{eq:mTf}, we see that this part has measure at least
        \[
            \frac{1}{C_f} \Xi_f(M) \ge \frac{(1-\sqrt[8]{\delta})^2 (1-\sqrt[4]{\delta})}{(1+\sqrt[4]{\delta})^2} \ge 1-10\sqrt[8]{\delta}.
        \]
        Hence, for any $f$ satisfying~\eqref{eq:f-bar} the total variation distance between the normalized measures does not exceed
        \[
            \TV(\mTn_{\alpha, \eps} (f_* \msp), \mTn_{\alpha, \eps}(\mgr * \msp))) \le 10\sqrt[8]{\delta}.
        \]
        As~\eqref{eq:f-bar} holds with the probability at least~$1-\frac{2\sqrt{\delta}}{1-\delta}> 1-4\sqrt{\delta}$, this establishes conclusion~\eqref{eq:Probably-TV}. Finally, this immediately implies conclusion~\eqref{eq:Probably} due to Lemma~\ref{lm:WasEst}.
    \end{proof}

    \vspace{3pt}

    We are now ready to conclude the proof of Proposition~\ref{prop:contr}. Namely, the conclusion~\eqref{eq:Probably} states that with high probability the measure $\mTn_{\alpha,\eps}(f_* \msp)$ is close to the deterministic one, $\mTn_{\alpha,\eps}(\mgr * \msp)$. The next step is to show that with high probability the first measure is close to $f_*$-image of a given measure,~$\mTn_{\alpha,\eps}(\msp)$.

    It easily follows from inequality (\ref{RegCondHol}) that we can choose $R < \infty$ (again, uniformly in $\mgr$), such that
    \begin{equation*}
        \p_{\mgr} \left[ L(f) < R \ \text{and } \gamma(f)^{-1} < R \right] > 1 - \sqrt[4]{\delta}
    \end{equation*}
    By Proposition \ref{prop:WassEst} we can choose $C$ big enough, so that for any $f \in \Hol(M)$, such that $L(f), \gamma(f)^{-1} < R$ inequality (\ref{tEnWassInv}) holds, and using triangle inequality we arrive to
    \begin{equation} \label{limitLaw}
        \p_{\mgr} \left[ W(f_* \mTn_{\alpha, \eps} (\msp), \mTn_{\alpha, \eps}(\mgr * \msp)) > 10 \diam(M) \sqrt[8]{\delta} + \delta \right] < 5 \sqrt[4]{\delta}.
    \end{equation}

    Now consider the sequence $\delta_n = \frac{1}{n}$ and denote by $\mgr_n$ and $\msp_n$ the corresponding measures and corresponding parameters by $\alpha_n, \eps_n$, for which (\ref{contrAssumption2}) holds. Then the measures $\mTn_{\alpha_n, \eps_n} (\msp_n)$ have a weakly convergent subsequence. Futhermore, we can take another subsequence, such that the measures $\mTn_{\alpha_{n_k}, \eps_{n_k}} (\mgr_{n_k} * \msp_{n_k})$ also converge. Passing to a subsequence one more time we can guarantee (due to compactness of $K$) that 
    \begin{equation*}
        \lim_{k \to \infty} \mgr_{n_k} = \mgr
    \end{equation*}
    for some $\mgr \in K$. It remains to notice that due to inequality (\ref{limitLaw}) the limit measure
    \begin{equation*}
        m = \lim_{k \to \infty} \mTn_{\alpha_{n_k}, \eps_{n_k}}(\msp_{n_k})
    \end{equation*}
    has an almost surely constant image under the action of $f \in \supp(\mgr)$, that is equal to
    \begin{equation*}
        \tilde{m} = \lim_{k \to \infty} \mTn_{\alpha_{n_k}, \eps_{n_k}} (\mgr_{n_k} * \msp_{n_k}) = \lim_{k \to \infty} \mTn_{\alpha_{n_k}, \eps_{n_k}} (\mgr * \msp_{n_k}),
    \end{equation*}
    which  contradicts our assumption.	
\end{proof}

\subsection{Proof of Theorem \ref{thm:mainHol-3}} \label{subsection:finalProof}

Now we have all the ingredients to implement the plan outlined in Section \ref{subsection:plan}. First we modify Proposition \ref{prop:contr} (contraction of energy) so that we have control over the image of any probability measure, not just those with high energies. Then we apply Markov's inequality (Lemma \ref{lm:EnLowerBound}) to finish the proof of Theorem \ref{thm:mainHol-3}.

We start by expanding Proposition \ref{prop:contr} so that it includes all measures on $M$. To do so, we need to adjust the upper bound~\eqref{EnContr} in order to include low-energy measures~$\msp$:

\begin{lemma} \label{l:uniform-C}
    Assume that the measure $\mgr$ satisfies the condition (\ref{RegCondHol}) with some $\beta$, and that $\alpha < \beta$ and $C$ are given. Then there exists $C'$ such that for any $\eps > 0$ and any measure $\msp$ on $M$ with $\En_{\alpha, \eps} (\msp) \le C$ one has
    \begin{equation*}
        \En_{\alpha, \eps} (\mgr*\msp) \le C'.
    \end{equation*}
\end{lemma}

\begin{corollary}\label{c:En}
    In the assumptions of Proposition~\ref{prop:contr} one can conclude that there exist $\alpha > 0, \tC < \infty, \lambda < 1$, such that for every $\mgr \in K$, every $\eps > 0$ and every measure $\msp$ on $M$
	\begin{equation} \label{eq:EnContr}
		\En_{\alpha, \eps} (\mgr * \msp) < \max(\lambda \En_{\alpha, \eps} (\msp),\tC).
	\end{equation}
\end{corollary}

\begin{proof}[Proof of Lemma~\ref{l:uniform-C}]

    Condition (\ref{RegCondHol}) implies the finiteness of the expectation
    \begin{equation}\label{eq:I-exp}
        C_I:= \E_{\mgr} \Lambda(f)^{\alpha} < \infty.
    \end{equation}
    Now,~\eqref{eq:avg-bound} implies that for every measure $\msp$ we have
    \[
        \tEn_{\alpha,\eps}(\mgr*\msp) \le \E_{\mgr} \tEn_{\alpha,\eps}(f_* \msp).
    \]
    At the same time Proposition~\ref{prop:EnEquiv} for $\delta=\frac{1}{2}$ implies that for some constant $C_1$ one has for any measure $\msp'$
    \[
        \tEn_{\alpha,\eps}(\msp') \le \max(2\En_{\alpha,\eps}(\msp'), C_1 ).
    \]
    Applying this for $\msp'=f_* \msp$ and joining it with Corollary~\ref{cor:EnOneStepForm2}, we get
    \begin{equation*}
        \tEn_{\alpha,\eps}(f_*\msp) \le \max(2\En_{\alpha,\eps}(f_*\msp), C_1 ) \le 2 \max \left( 2 \Lambda(f)^{\alpha} \En_{\alpha, \eps} (\msp), C_2 \right)+ C_1,
    \end{equation*}
    where $C_2$ is equal to $C$ from Corollary (\ref{cor:EnOneStepForm2}). Taking the expectation w.r.t.~$\mgr$ and using~\eqref{eq:I-exp}, we finally get a uniform bound
    \[
        \tEn_{\alpha,\eps}(\mgr*\msp) \le 4 C_I C + 2 C_2 + C_1 =: C'.
    \]

\end{proof}

\begin{proof}[Proof of Corollary~\ref{c:En}]
    In the proof of Proposition~\ref{prop:contr} we can take $\alpha$ arbitrarily small, so we can assume that $\alpha < \beta$. Applying Proposition~\ref{prop:EnEquiv} to both sides of~\eqref{EnContr}, where we take $\delta$ sufficiently small so that $\frac{1-\delta}{1+\delta}>\lambda$, allows to conclude that there exists some constant $C_3$ such that for all $\eps$ and $\msp$
    \[
        \En_{\alpha, \eps} (\mgr * \msp) < \lambda' \En_{\alpha, \eps} ( \msp) \quad \text{ if } \, \En_{\alpha, \eps} ( \msp) > C_3,
    \]
    where $\lambda'=\frac{1+\delta}{1-\delta} \lambda <1$. Meanwhile, Lemma~\ref{l:uniform-C} implies that there exists some $C_4$ such that
    \[
        \En_{\alpha, \eps} (\mgr * \msp) <C_4 \quad \text{ if } \, \En_{\alpha, \eps} ( \msp) \le C_3.
    \]
    Taking $\tC := C_4$, we get the desired~\eqref{eq:EnContr}.
\end{proof}

Now we are ready to complete the proof of Theorem~\ref{thm:mainHol-3} (and hence of Theorem \ref{thm:mainHol-1}).

\begin{proof}[Proof of Theorem~\ref{thm:mainHol-3}]

    Let $\alpha,\tC$ be as in Corollary~\ref{c:En}: for any $\mgr \in K$, any $\eps > 0$ and for any measure $\msp$ one has
    \[
        \En_{\alpha, \eps} (\mgr * \msp) < \max(\lambda \En_{\alpha, \eps} (\msp),\tC);
    \]
    applying this $n$ times, we get
    \[
        \En_{\alpha, \eps} (\mgr^n * \msp) < \max (\lambda^n \En_{\alpha, \eps} (\msp),\tC).
    \]
    Recall that Lemma~\ref{lm:EnUpperBound} gives a uniform upper bound 
    \begin{equation*}
        \En_{\alpha, \eps} (\msp) \le (\omega_k + c_{\alpha, k}) (- \log(\eps))^{\alpha},
    \end{equation*}
    choose
    \begin{equation*}
        \eps := e^{- \left( \frac{1}{\lambda} \right)^{\frac{n}{\alpha}}},
    \end{equation*}
    then we have
    \begin{equation*}
        \En_{\alpha, \eps} (\mgr^n * \msp) < \max \left( (\omega_k + c_{\alpha, k}), \tC \right).
    \end{equation*}
    Now, applying Lemma~\ref{lm:EnLowerBound} for any $r > \eps = e^{-\kappa^n}$, where $\kappa := (1/\lambda)^{1/\alpha} > 1$, we deduce that
    \begin{equation*}
        (\mgr^n * \msp) (B_r(x)) < C_{\alpha, k} \sqrt{ \max \left( (\omega_k + c_{\alpha, k}), \tC \right)} \cdot (- \log(r))^{-\frac{\alpha}{2}},
    \end{equation*}
    which completes the proof of Theorem \ref{thm:mainHol-3}.

\end{proof}



\section{Lipschitz homeomorphisms} \label{sec:LipProof}

In this section we will outline the proof of Theorem \ref{thm:mainLip-3}. Since it utilizes the same method as the proof of Theorem \ref{thm:mainHol-3} we will give modified versions of key propositions avoiding technical details. In this section we will work with a fixed $\alpha > 0$.

First, we will need an analog of Proposition \ref{prop:EnChangeHol} to have an estimate for $\En_{\alpha, \eps}(f_* \msp)$ for a Lipschitz continuous $f$. It can be formulated as follows:

\begin{proposition}
    For every $\delta > 0$ there exists $\eps_0 > 0$ and $B < \infty$, such that every $0 < \eps < \eps_0$, every $f \in \Lip(M)$ and every $\msp \in \mM$ the following holds:
    \begin{equation} \label{EnLipOneStep}
            \En_{\alpha, \eps}(f_* \msp) \le (1 + \delta \log(\fL(f)))^{\alpha} (\En_{\alpha, \eps}(\msp) + B).
        \end{equation}
\end{proposition}


In this case useful corollaries can be formulated as follows:

\begin{corollary} \label{cor:EnLip}
    For every $\delta > 0$ there exists $C < \infty$ such that for every $\eps > 0$, every $f \in \Lip(M)$ and every $\msp \in \mM$, such that $\En_{\alpha, \eps} (\msp) > C$ the following holds:
    \begin{equation} \label{EnLipOneStepCor}
        \frac{1}{1 + \delta} (1 + \delta \log ( \fL(f) ))^{-\alpha} < \frac{\En_{\alpha, \eps}(f_* \msp)}{\En_{\alpha, \eps}(\msp)} < (1 + \delta) (1 + \delta \log ( \fL(f) ))^{\alpha}.
    \end{equation}
\end{corollary}

\begin{corollary}
    There exists $C < \infty$ such that for every $\eps > 0$, every $f \in \Lip(M)$ and every $\msp \in \mM$ the following holds:
    \begin{equation*}
        \En_{\alpha, \eps}(f_* \msp) < \max ( 2 (1 + \log ( \fL(f) ))^{\alpha} \En_{\alpha, \eps}(\msp), C ).
    \end{equation*}
\end{corollary}

Notice, that Proposition \ref{prop:EnEquiv} does not depend on the regularity of $f$, hence it remains unchanged. That, once again, allows us to establish an estimate for $\tEn_{\alpha, \eps} (f_* \msp)$ (analog of Corollary \ref{cor:tEnChangeOneStepHol}):

\begin{corollary}
    For every $\delta > 0$ and every $R < \infty$ there exists $C > 0$ such that for every $f \in \Lip(M)$ with $\fL(f) < R$, every $\eps > 0$ and every $\msp$ such that $\tEn_{\alpha, \eps}(\msp) > C$ or $\En_{\alpha, \eps}(\msp) > C$ one has
    \begin{equation*}
        \frac{\tEn_{\alpha, \eps} (f_*\msp)}{\tEn_{\alpha, \eps} (\msp)} \in (1-\delta, 1+\delta).
    \end{equation*}
\end{corollary}

Next on the list is the analog of Proposition \ref{prop:WassEst}, which is formulated bellow.

\begin{proposition}
    For any $\delta > 0$, any $R < \infty$ and any $\alpha \in (0, 1]$ there exists $C > 0$ such that for any $\eps > 0$, any $f \in \Lip(M)$, such that $\fL(f) < R$ and any $\msp \in \mM$ such that $\En_{\alpha, \eps} (\msp) > C$ or $\tEn_{\alpha, \eps} (\msp) > C$ one has
    \begin{equation*}
        W(f_* \theta_{\alpha, \eps} [\msp], \theta_{\alpha, \eps} [f_* \msp]) < \delta.
    \end{equation*}
\end{proposition}

Once again, the proof repeats the one of Proposition \ref{prop:WassEst} up to minor computational details. Finally, the uniform estimate (\ref{RegCondLip}) allows us to establish the following analog of Proposition \ref{lm:tailsEst}:

\begin{proposition}
    If a probability measure $\mgr$ on $\Lip(M)$ satisfies assumption (\ref{RegCondLip}) then for every $\delta > 0$ there exists $C > 0$ such that for every $\eps > 0$ and every measure $\msp \in \mM$ such that $\En_{\alpha, \eps} (\msp) > C$ the following formula holds:
    \begin{equation*}
        \frac{\E_{\mgr} \left[ \tEn_{\alpha, \eps} (f_* \msp) \right]}{\tEn_{\alpha, \eps} (\msp)} \in (1 - \delta, 1 + \delta).
    \end{equation*}
\end{proposition}

Utilizing these modified statements we are able to repeat the proof of the contracting property for $\tEn_{\alpha, \eps}$ (analog of Proposition \ref{prop:contr}, but with a fixed $\alpha$ taken from assumption \eqref{AssumptionLip}):

\begin{proposition} \label{prop:LipContr}
    Under the assumptions of Theorem \ref{thm:mainLip-1} there exists $\lambda < 1$ and $C < \infty$, such that for every $\mgr \in K$, every $\eps > 0$ and every $\msp \in \mM$ if $\En_{\alpha, \eps} (\msp) > C$ then
    \begin{equation*}
        \tEn_{\alpha, \eps} (\mgr * \msp) < \lambda \tEn_{\alpha, \eps} (\msp).
    \end{equation*}
\end{proposition}

Finally, we can formulate a similar estimate for all measures and not just those with high energy. This statement will also be used in Appendix \ref{appendix:B} to prove a technical generalization of Theorem \ref{thm:mainLip-3}, which was mentioned at the end of Section \ref{subsection:applications}. 
\begin{corollary} \label{Lip:EnUniv}
    In the assumptions of Proposition~\ref{prop:LipContr} one can conclude that there exist $\tC < \infty, \lambda < 1$, such that for every $\mgr \in K$, every $\eps > 0$ and every measure $\msp$ on $M$
	\begin{equation*} 
		\En_{\alpha, \eps} (\mgr * \msp) < \max(\lambda \En_{\alpha, \eps} (\msp),\tC).
	\end{equation*}
\end{corollary}

From here Theorem \ref{thm:mainLip-3} (and hence Theorem \ref{thm:mainLip-1}) follows.

\appendix
\section{Technical lemmata} \label{appendix}

In this section we establish some properties of functions $\varphi_{\alpha}$ and $\varphi_{\alpha, \eps}$ and then prove Proposition \ref{prop:UAEps}.

A straightforward computation allows us to prove the following
\begin{lemma} \label{lm:partInt}
    For any $0 < r < \frac{1}{e}$
    \begin{equation*}
        \int_{\bx \in \R^k, r < |\bx| < \frac{1}{e}} \varphi_{\alpha} (|\bx|)^2  \dd \bx = c_{\alpha, k} ((-\log(r))^{\alpha} - 1),
    \end{equation*}
    where $c_{\alpha, k} = \frac{k \omega_k}{\alpha}$, and $\omega_k$ is the volume of a unit ball in $\R^k$.
\end{lemma}

\begin{proof}
    Since the function $\varphi_{\alpha} (|\bx|)$ depends only on the radius $|\bx|$ we conclude that its integral over a sphere of radius $y$ is equal to
    \begin{equation*}
        \int_{\bx \in \R^k, |\bx| = y} \varphi_{\alpha} (|\bx|)^2 \dd S(\bx) = k \omega_k y^{k - 1} \varphi_{\alpha} (y)^2
    \end{equation*}
    ($\omega_k$ is the volume of a unit ball in $\R^k$, hence $k \omega_k y^{k - 1}$ is the $(k - 1)$-dimensional volume of a sphere of radius $y$). Integrating over the radius $y$ we get:
    \begin{multline*}
        \int_{\bx \in \R^k, r < |\bx| < \frac{1}{e}} \varphi_{\alpha} (|\bx|)^2  \dd \bx = \int_{r}^{1/e} k \omega_k y^{k - 1} \frac{|\log(y)|^{\alpha - 1}}{y^k} \dd y = \\ 
        = k \omega_k \int_{r}^{1/e} \frac{(-\log(y))^{\alpha - 1}}{y} \dd y = \frac{k \omega_k}{\alpha} (|\log(r)|^{\alpha} - 1).
    \end{multline*}
\end{proof}

Let us define a function 
\begin{equation*}
    V_{\alpha, \eps} (r) = \begin{cases}
        c_{\alpha, k} |\log(r)|^{\alpha}, \quad \text{for $\eps < r < 1/e$;} \\
        c_{\alpha, k} |\log(\eps)|^{\alpha}, \quad \text{for $r \le \eps$.}
    \end{cases}
\end{equation*}
Our next goal is to prove
\begin{proposition} \label{prop:singEst}
    For every $\alpha_0 > 0$ and every $\delta > 0$ there exists $\eps_0 > 0$ and $r_0 > 0$, such that for every $0 < \alpha < \alpha_0$, for every $0 < \eps < \eps_0$, and every $0 \le r < r_0$ we have
    \begin{equation*}
        \frac{1}{1 + \delta} < \frac{U_{\alpha, \eps} (r)}{V_{\alpha, \eps} (r)} < 1 + \delta,
    \end{equation*}
    where the function $U_{\alpha, \eps} (r)$ is defined by equation \eqref{def:UAE}.
\end{proposition}

\begin{proof}

Given $\alpha_0 > 0$ and $\delta > 0$, take some constant $c > 0$; its exact value will be chosen later (and will depend on $\delta$ and $\alpha_0
$, but not $r$). Now, write the integral \eqref{def:UAE}:
\begin{equation*}
    U_{\alpha, \eps} (r) = \int_{\R^k} \varphi_{\alpha, \eps} (|\bx|) \varphi_{\alpha, \eps} (|\bx - \br|) \dd \bx,
\end{equation*}
splitting the domain of integration into several parts:
\begin{enumerate}[label=(\roman*)]
    \item \label{balls} Balls $B_{r/2} (0)$ and $B_{r/2} (\br)$;
    \item \label{difference} Difference $B_{c r} (0) \setminus (B_{r/2} (0) \cup B_{r/2} (\br))$;
    \item \label{outer} Outer region $\R^k \setminus B_{\frac{1}{2e}} (0)$;
    \item Spherical layer between two spheres $B_{\frac{1}{2e}} (0) \setminus B_{c r} (0)$.
\end{enumerate}
We will show that, for an appropriate choice of the constant $c$, $\eps_0$, and $r_0$ the integrals over all these regions except for the last one do not exceed $\frac{\delta}{10} V_{\alpha, \eps} (r)$, and the integral over the spherical layer will differ from $V_{\alpha, \eps} (r)$ by at most $\frac{\delta}{10} V_{\alpha, \eps} (r)$:
\begin{equation} \label{mainSph}
    1 - \frac{\delta}{10} < \frac{1}{V_{\alpha, \eps} (r)} \int_{c r < |\bx| < \frac{1}{2e}} \varphi_{\alpha, \eps} (|\bx|) \varphi_{\alpha, \eps} (|\bx - \br|) \dd \bx < 1 + \frac{\delta}{10}.
\end{equation}

Adding such upper bounds once they are established would give
\begin{equation*}
    1 - \frac{\delta}{2} < \frac{U_{\alpha, \eps} (r)}{V_{\alpha, \eps} (r)} < 1 + \frac{\delta}{2},
\end{equation*}
so that would conclude the proof.

Let us start by establishing the required inequality for the outer region. If $|\bx| > c r$ then
\begin{equation*}
    1 - \frac{1}{c} < \frac{|\bx - \br|}{|\bx|} < 1 + \frac{1}{c}.
\end{equation*}
Hence, for $c$ big enough we have
\begin{equation*}
    - \frac{2}{c} < \log|\bx - \br| - \log|\bx| < \frac{1}{c},
\end{equation*}
and for any $cr<|\bx| < 1/e$ we get
\begin{equation*}
    1 - \frac{2}{c} < \frac{\log |\bx - \br|}{\log |\bx|} < 1 + \frac{1}{c}.
\end{equation*}
Recalling the definition of $\varphi_{\alpha, \eps}$ (Definition \ref{def:phiae}) we conclude that given $\delta > 0$ and $\alpha_0$ we can pick $r_0 < \frac{1}{2e}$ and $c$ big enough so that
\begin{equation*}
    1 - \frac{\delta}{100} < \frac{\varphi_{\alpha, \eps} (|\bx|) \varphi_{\alpha, \eps} (|\bx - \br|)}{\varphi_{\alpha, \eps} (|\bx|)^2} < 1 + \frac{\delta}{100}
\end{equation*}
for every $\bx$ in spherical layer $B_{\frac{1}{2e}} (0) \setminus B_{c r} (0)$. Integrating this inequality we get
\begin{equation} \label{sphIntEq}
    1 - \frac{\delta}{100} < \frac{\int_{cr < |\bx| < \frac{1}{2e}} \varphi_{\alpha, \eps} (|\bx|) \varphi_{\alpha, \eps} (|\bx - \br|) \dd  \bx}{\int_{cr < |\bx| < \frac{1}{2e}} \varphi_{\alpha, \eps} (|\bx|)^2 \dd  \bx} < 1 + \frac{\delta}{100}.
\end{equation}
Using Lemma \ref{lm:partInt} it's not too hard to establish that given $\alpha_0 > 0$, $\delta > 0$, and $c$ we can choose $r_0 > 0$ and $\eps_0 > 0$ so that for every $0 < \alpha < \alpha_0$, every $0 < \eps < \eps_0$, and $0 < r < r_0$ we have 
\begin{equation} \label{VEquiv}
    1 - \frac{\delta}{100} < \frac{\int_{cr < |\bx| < \frac{1}{2e}} \varphi_{\alpha, \eps} (|\bx|)^2 \dd \bx}{V_{\alpha, \eps} (r)} < 1 + \frac{\delta}{100}.
\end{equation}
Namely, to prove the inequality above one needs to treat cases $r \ge \eps$ and $r < \eps$ separately. If $r \ge \eps$ then by Lemma \ref{lm:partInt} we have
\begin{multline*}
    V_{\alpha, \eps} (r) = \int_{r < |\bx| < 1/e} \varphi_{\alpha, \eps} (|\bx|)^2 \dd \bx + c_{\alpha, k} = \\
    = \int_{r < |\bx| < cr} \varphi_{\alpha, \eps} (|\bx|)^2 \dd \bx + \int_{cr < |\bx| < \frac{1}{2e}} \varphi_{\alpha, \eps} (|\bx|)^2 \dd \bx + \int_{\frac{1}{2e} < |\bx| < \frac{1}{e}} \varphi_{\alpha, \eps} (|\bx|)^2 \dd \bx + c_{\alpha, k} = \\
    = c_{\alpha, k} ((- \log(r))^{\alpha} - (- \log(cr))^{\alpha}) + c_{\alpha, k} (1 + \log(2))^{\alpha} + \int_{cr < |\bx| < \frac{1}{2e}} \varphi_{\alpha, \eps} (|\bx|)^2 \dd \bx.
\end{multline*}
Now we notice that all summands except for the last one are small compared to $V_{\alpha, \eps}(r) = c_{\alpha, k} (- \log(r))^{\alpha}$ for $r$ close to zero, i. e.
\begin{equation*}
    \frac{c_{\alpha, k} ((- \log(r))^{\alpha} - (- \log(cr))^{\alpha}) + c_{\alpha, k} (1 + \log(2))^{\alpha}}{c_{\alpha, k} (- \log(r))^{\alpha}} \to_{r \to 0} 0,
\end{equation*}
and the convergence is uniform in $\alpha \in (0, \alpha_0)$. Hence, by taking small enough $r_0$ we guarantee \eqref{VEquiv}. If $r < \eps$ we can represent $V_{\alpha, \eps} (r)$ in the following way:
\begin{equation*}
    V_{\alpha, \eps} (r) = V_{\alpha, \eps} (\eps) = \int_{\eps < |\bx| < \frac{1}{e}} \varphi_{\alpha, \eps} (|\bx|)^2 \dd \bx + c_{\alpha, k}.
\end{equation*}
The domain of that integral intersects with $(cr, \frac{1}{2e})$ by at least $(c \eps, \frac{1}{2e})$ (since $c > 1$ and $r < \eps$). Hence the following estimate holds:
\begin{multline*}
    \left| \int_{cr < |\bx| < \frac{1}{2e}} \varphi_{\alpha, \eps} (|\bx|)^2 \dd \bx - V_{\alpha, \eps} (r) \right| \le \\
    \le \int_{0 < |\bx| < \eps} \varphi_{\alpha} (\eps)^2 \dd \bx + \int_{\eps < |\bx| < c \eps} \varphi_{\alpha} (|\bx|)^2 \dd \bx + \int_{\frac{1}{2e} < |\bx| < \frac{1}{e}} \varphi_{\alpha} (|\bx|)^2 \dd \bx + c_{\alpha, k} = \\
    = \omega_k \eps^k \frac{(-\log(\eps))^{\alpha - 1}}{\eps^k} + c_{\alpha, k} ((- \log(\eps))^{\alpha} - (- \log(c\eps))^{\alpha}) + c_{\alpha, k} (1 + \log(2)).
\end{multline*}
Dividing both sides by $V_{\alpha, \eps}(r) = c_{\alpha, k} (-\log(\eps))^{\alpha}$ and picking $\eps_0$ small enough we establish \eqref{VEquiv}.

Now equations \eqref{sphIntEq} and \eqref{VEquiv} together imply \eqref{mainSph}.

We organize the rest of the proof of Proposition \ref{prop:singEst} in a sequence of Lemmata estimating integrals over domains mentioned in \ref{balls}, \ref{difference}, and \ref{outer}.

\begin{lemma}
    For every $\alpha_0 > 0$, and $\delta > 0$ there exists $r_0 > 0$ and $\eps_0 > 0$ such that for every $0 < \alpha < \alpha_0$, $0 < \eps < \eps_0$, and every $0 < r < r_0$ we have
    \begin{equation} \label{ballsIneq}
        \int_{B_{r/2} (0) \cup B_{r/2} (\br)} \varphi_{\alpha, \eps} (|\bx|) \varphi_{\alpha, \eps} (|\bx - \br|) \dd \bx < \frac{\delta}{10} V_{\alpha, \eps} (r).
    \end{equation}
\end{lemma}

\begin{proof}
    Due to the central symmetry of $\varphi_{\alpha, \eps} (|\bx|)$ it will be enough to estimate just the integral over $B_{r/2} (0)$. If $r < \eps$ as we have seen before 
    \begin{equation*}
        \int_{B_{r/2} (0)} \varphi_{\alpha, \eps} (|\bx|) \varphi_{\alpha, \eps} (|\bx - \br|) \dd \bx < \int_{|\bx| < \eps} \varphi_{\alpha, \eps} (\eps)^2 \dd \bx = \omega_k (-\log(\eps))^{\alpha - 1}
    \end{equation*}
    and for small enough $\eps_0$ the estimate \eqref{ballsIneq} follows. 
    
    If $r \ge \eps$ the desired inequality will follow from the uniform in $\alpha \in (0, \alpha_0)$ convergence 
    \begin{equation*}
        \frac{\int_{|\bx| < r/2} \varphi_{\alpha} (|\bx|) \varphi_{\alpha} (|\br - \bx|) \dd \bx}{(- \log(r))^{\alpha}} \to_{r \to 0} 0.
    \end{equation*}
    To prove said convergence we start with the following computation:
    \begin{multline} \label{ineq:ball}
        \int_{|\bx| < r/2} \varphi_{\alpha} (|\bx|) \varphi_{\alpha} (|\br - \bx|) \dd \bx < \varphi_{\alpha} (r/2) \int_{|\bx| < r/2} \varphi_{\alpha} (|\bx|) \dd \bx = \\ 
        = k \omega_k \frac{(- \log(r/2))^{\frac{\alpha - 1}{2}}}{(r/2)^{k/2}} \int_{0}^{r/2} \frac{(- \log(t))^{\frac{\alpha - 1}{2}}}{t^{k/2}} t^{k - 1} \dd t.
    \end{multline}

    We will focus on the last integral. If $\alpha \le 1$ then 
    \begin{equation*}
        \int_{0}^{r/2} (- \log(t))^{\frac{\alpha - 1}{2}} t^{k/2 - 1} \dd t \le \int_{0}^{r/2} t^{k/2 - 1} \dd t = \frac{2}{k} \left( \frac{r}{2} \right)^{k/2}.
    \end{equation*}
    Combining that with formula (\ref{ineq:ball}) we get
    \begin{equation*}
        \frac{\int_{|\bx| < r/2} \varphi_{\alpha} (|\bx|) \varphi_{\alpha} (|\br - \bx|) \dd \bx}{(-\log(r))^{\alpha}} < 2 \omega_k \frac{(- \log(r/2))^{\frac{\alpha - 1}{2}}}{(- \log(r))^{\alpha}}
    \end{equation*}
    and uniform convergence follows.

    If $\alpha > 1$ after taking the last integral in formula (\ref{ineq:ball}) by parts we obtain:
    \begin{multline} \label{eq:intByParts}
        \int_{0}^{r/2} (- \log(t))^{\frac{\alpha - 1}{2}} t^{k/2 - 1} \dd t = \\
        = \left. \frac{2 (- \log(t))^{\frac{\alpha - 1}{2}} t^{k/2}}{k} \right|_{0}^{r/2} + \frac{\alpha - 1}{k} \int_{0}^{r/2} (- \log(t))^{\frac{\alpha - 3}{2}} t^{k/2 - 1} \dd t.
    \end{multline}
    Since $r$ goes to zero we can assume that $- \log(r/2) > 2 (\alpha_0 - 1)/k$. In that case we can write
    \begin{multline*}
        \frac{\alpha - 1}{k} \int_{0}^{r/2} (- \log(t))^{\frac{\alpha - 3}{2}} t^{k/2 - 1} \dd t = \\
        = \frac{\alpha - 1}{k} \int_{0}^{r/2} (- \log(t))^{\frac{\alpha - 3}{2}} (- \log(t)) ( - \log(t))^{-1} t^{k/2 - 1} \dd t < \\
        < \frac{\alpha - 1}{k} (-\log(r/2))^{-1} \int_{0}^{r/2} (- \log(t))^{\frac{\alpha - 1}{2}} t^{k/2 - 1} \dd t < \\
        < \frac{1}{2} \int_{0}^{r/2} (- \log(t))^{\frac{\alpha - 1}{2}} t^{k/2 - 1} \dd t.
    \end{multline*}
    Together with formula (\ref{eq:intByParts}) the last inequality implies 
    \begin{equation*}
        \int_{0}^{r/2} (- \log(t))^{\frac{\alpha - 1}{2}} t^{k/2 - 1} \dd t < \frac{4}{k} \left( - \log \left( \frac{r}{2} \right)\right)^{\frac{\alpha - 1}{2}} \left( \frac{r}{2} \right)^{k/2}
    \end{equation*}
    and the uniform convergence follows.
\end{proof}

\begin{lemma}
    For every $\alpha_0 > 0$, $\delta > 0$, and $c > 0$ there exists $r_0 > 0$ and $\eps_0>0$ such that for every $0 < \alpha < \alpha_0$, $0 < \eps < \eps_0$, and every $0 < r < r_0$ we have
    \begin{equation} \label{diffIneq}
        \int_{B_{cr}(0) \setminus (B_{r/2} (0) \cup B_{r/2} (\br))} \varphi_{\alpha, \eps} (|\bx|) \varphi_{\alpha, \eps} (|\bx - \br|) \dd \bx < \frac{\delta}{10} V_{\alpha, \eps} (r).
    \end{equation}
\end{lemma}

\begin{proof}
    Let us denote the domain of integration by $D$:
    \begin{equation*}
        D = B_{cr} (0) \setminus (B_{r/2} (0) \cup B_{r/2} (\br)).
    \end{equation*}
    Using Cauchy-Schwarz inequality we conclude that
    \begin{equation*}
        \int_{D} \varphi_{\alpha, \eps} (|\bx|) \varphi_{\alpha, \eps} (|\bx - \br|) \dd  \bx \le \sqrt{ \int_{D} \varphi_{\alpha, \eps} (|\bx|)^2 \dd  \bx \int_{D} \varphi_{\alpha, \eps} (|\bx - \br|)^2 \dd  \bx}. 
    \end{equation*}
    By increasing domains of both integrals we arrive to
    \begin{equation*}
        \int_{D} \varphi_{\alpha, \eps} (|\bx|) \varphi_{\alpha, \eps} (|\bx - \br|) \dd \bx \le \int_{r/2 < |\bx| < 2cr} \varphi_{\alpha, \eps} (|\bx|)^2 \dd \bx.
    \end{equation*}
    Now we have to consider two cases, depending on the relation of $r$ and $\eps$. We treat these cases similarly to our proof of estimate \eqref{VEquiv}. Assume that $r \ge \eps$. Then by Lemma \ref{lm:partInt} we have
    \begin{equation*}
        \int_{r/2 < |\bx| < 2cr} \varphi_{\alpha, \eps} (|\bx|)^2 \dd \bx \le c_{\alpha, k} ((-\log(r/2))^{\alpha} - (-\log(cr))^{\alpha}).
    \end{equation*}
    Choosing $r_0$ small enough we guarantee inequality \eqref{diffIneq}.

    If $r < \eps$ we have 
    \begin{equation*}
        \int_{r/2 < |\bx| < 2cr} \varphi_{\alpha, \eps} (|\bx|)^2 \dd \bx < \omega_k \eps^k \frac{|\log(\eps)|^{\alpha - 1}}{\eps^k} + \int_{\eps < |\bx| < 2c\eps} \varphi_{\alpha, \eps} (|\bx|)^2 \dd \bx,
    \end{equation*}
    and choosing $\eps_0$ small enough we establish \eqref{diffIneq}.

\end{proof}

To conclude the proof of Proposition \ref{prop:singEst} it remains to notice that the integral over domain \ref{outer} is uniformly bounded. 

\end{proof}

Finally, using Proposition \ref{prop:singEst} we are able to prove Proposition \ref{prop:UAEps}

\begin{proof}[Proof of Proposition \ref{prop:UAEps}]

    \begin{enumerate}

        \item[\textbf{(I)}] First, notice that the function $\varphi_{\alpha, \eps}(r)$ is non\-increasing. Using that fact Part \textbf{(I)} follows from the definition of $U_{\alpha, \eps}$, for details see \cite[Lemma 6.6]{GKM}.

        \item[\textbf{(II)}] Choosing $\delta = 1$ in Proposition \ref{prop:singEst} we obtain inequality (\ref{ineq:UB}).

        \item[\textbf{(III)}] Choosing $\delta = 1$ in Proposition \ref{prop:singEst} we obtain inequality (\ref{ineq:LB}).

        \item[\textbf{(IV)}] Thanks to Proposition \ref{prop:singEst} it is enough to prove inequality (\ref{UOneStep}) for $V_{\alpha, \eps} (r_1)$ and $V_{\alpha, \eps} (r_2)$, which easily follows from the definition.
    \end{enumerate}

\end{proof}

\section{Joint regularity of two random measures} \label{appendix:B}

In this section we present a tailored version of Theorem \ref{thm:mainLip-3} that we intend to use in order to prove Central Limit Theorem for non-stationary products of random $\SL(2, \R)$ matrices with optimal assumption on the the distributions' tails (see \cite{GKM2}). Namely, the following result allows us to control the probability that two points independently generated by iterating two random dynamical systems will end up close to one another.

\begin{theorem} \label{thm:forCLT}
    Let $K$ be a compact set (with respect to weak-* topology) in the space of Borel probability measures on $\Homeo(M)$, satisfying the following assumptions:
    \begin{itemize}
        \item For every $\mgr \in K$ we have $\supp(\mgr) \subset \Lip(M)$.
        \item 
        There exist $\alpha > 0$ and $C_0$ such that for every $\mgr \in K$
        \begin{equation} \label{CLT_const}
            \int_{\Lip(M)} \left( \log(\fL(f)) \right)^{\alpha} \; d\mgr(f) < C_0.
        \end{equation}
        \item For every $\mgr \in K$ there are no measures $m_1, m_2 \in \mM$, such that $f_*m_1=m_2$ for $\mgr$-a.e.~$f$.
    \end{itemize}
    Then there exist $C$ and $\kappa > 1$ such that for any initial measures $\msp_0^{(1)}, \msp_0^{(2)}$, every number of iterations $n_1, n_2 \in \N$, and every choice of two sequences of measures $\mgr_1^{(1)}, \mgr_2^{(1)}, \ldots, \mgr_{n_1}^{(1)} \in K$ and $\mgr_1^{(2)}, \mgr_2^{(2)}, \ldots, \mgr_{n_2}^{(2)} \in K$ one has:
    \begin{equation} \label{CLT_est}
        \iint_{M \times M} |\log \left(\max\{d(x, y), r\}\right)|^{\alpha} \dd \msp_1(x) \dd \msp_2(y) < C,
    \end{equation}
    where $r = e^{-\kappa^{\min(n_1, n_2)}}$, $\msp_1 = \mgr_{n_1}^{(1)} * \mgr_{n_1 - 1}^{(1)} * \ldots * \mgr_1^{(1)} * \msp_0^{(1)}$, and $\msp_2 = \mgr_{n_2}^{(2)} * \mgr_{n_2 - 1}^{(2)} * \ldots * \mgr_{1}^{(2)} * \msp_0^{(2)}$.

\end{theorem}

\begin{proof}

    For this proof let us fix $k = \dim(M)$ and $\alpha > 0$ from inequality \eqref{CLT_const}. According to Part \textbf{(III)} of Proposition \ref{prop:UAEps} there exists $\eps_0 > 0$ and $C_U$ such that for any $\eps < r < \eps_0$ we have
    \begin{equation*}
        U_{\alpha, \eps}(r) > C_U |\log(r)|^{\alpha}.
    \end{equation*}
    Fix $\kappa = \lambda^{-\frac{1}{\alpha}}$, some $r > e^{-\kappa^{\min(n_1, n_2)}}$, and measures $\msp_1, \msp_2$ from the statement. In order to prove \eqref{CLT_est} it is enough to show that for some constant $\tilde{C}$ we have
    \begin{equation*}
        \iint_{M \times M} U_{\alpha, r}(d(x, y)) \dd \msp_1(x) \dd \msp_2(y) < \tilde{C}.
    \end{equation*}
    Notice that according to the Definition \ref{def:En}:
    \begin{multline*}
        \En_{\alpha, r} \left( \frac{\msp_1 + \msp_2}{2} \right) = \iint_{M \times M} U_{\alpha, r}(d(x, y)) \dd \frac{\msp_1 + \msp_2}{2}(x) \dd \frac{\msp_1 + \msp_2}{2}(y) \ge \\ 
        \ge \frac{1}{2}\iint_{M \times M} U_{\alpha, r}(d(x, y)) \dd \msp_1(x) \dd \msp_2(y),
    \end{multline*}
    so all we need is to come up with a constant upper bound for $\En_{\alpha, r} \left( \frac{\msp_1 + \msp_2}{2} \right)$. Thanks to Proposition \ref{prop:EnEquiv} it is sufficient to estimate $\tEn_{\alpha, r} \left( \frac{\msp_1 + \msp_2}{2} \right)$. Now we are able to employ a geometric inequality in $L^2(M, \Leb)$ using vectors $\rho_{\alpha, r} [\msp_1]$ and $\rho_{\alpha, r} [\msp_2]$. Namely, recall that due to the Definition \ref{def:tEn} we have
    \begin{multline*}
        \tEn_{\alpha, r} \left( \frac{\msp_1 + \msp_2}{2} \right) = \left\| \frac{\rho_{\alpha, r} [\msp_1] + \rho_{\alpha, r} [\msp_2]}{2} \right\|_{L^2(M)}^2 \le \\ 
        \le \frac{\|\rho_{\alpha, r} [\msp_1]\|_{L^2(M)}^2 + \|\rho_{\alpha, r} [\msp_1]\|_{L^2(M)}^2}{2} = \frac{1}{2} (\tEn_{\alpha, r} (\msp_1) + \tEn_{\alpha, r} (\msp_2)).
    \end{multline*}
    Using Proposition \ref{prop:EnEquiv} one more time we conclude that for big enough energies we have 
    \begin{equation*}
        \frac{1}{2} (\tEn_{\alpha, r} (\msp_1) + \tEn_{\alpha, r} (\msp_2)) \le \En_{\alpha, r} (\msp_1) + \En_{\alpha, r} (\msp_2).
    \end{equation*}
    Starting from here we are mimicking the end of the proof of Theorem \ref{thm:mainHol-3}. Applying Corollary \ref{Lip:EnUniv} we arrive to 
    \begin{equation*}
        \En_{\alpha, r} (\msp_1) \le \max \left( \lambda^{n_1} \En_{\alpha, r} \left( \msp_0^{(1)} \right), \tC \right) \quad \text{and} \quad \En_{\alpha, r} (\msp_2) \le \max \left( \lambda^{n_2} \En_{\alpha, r} \left( \msp_0^{(2)} \right), \tC \right).
    \end{equation*}
    From Lemma \ref{lm:EnUpperBound} we know that
    \begin{equation*}
        \max \left( \En_{\alpha, r} \left( \msp_0^{(1)} \right), \En_{\alpha, r} \left( \msp_0^{(2)} \right) \right) < (\omega_k + c_{\alpha, k}) (- \log(r))^{\alpha},
    \end{equation*}
    so since 
    \begin{equation*}
        r > e^{-\kappa^{\min(n_1, n_2)}} = e^{-\lambda^{- \frac{\min(n_1, n_2)}{\alpha}}}    
    \end{equation*}
    we have 
    \begin{equation*}
        \En_{\alpha, r} (\msp_1) + \En_{\alpha, r} (\msp_2) < \max \left( 2 (\omega_k + c_{\alpha, k}), \tC \right).
    \end{equation*}
    From here Theorem \ref{thm:forCLT} follows. 

\end{proof}

\section*{Acknowledgements}

The author is grateful to A. Gorodetski and V. Kleptsyn for fruitful discussions and to anonymous reviewers for their advice on improving both contents and presentation of the paper. The author was supported in part by NSF grant DMS--2247966 (PI: A.\,Gorodetski).


\begin{thebibliography}{99}

\bibitem[A]{A}  L.\,Arnold, Random dynamical systems. {\it Dynamical systems (Montecatini Terme, 1994)}, pp. 1--43, Lecture Notes in Math., {\bf 1609}, Springer, Berlin, 1995.


\bibitem[BBS]{BBS} V. Baladi, M. Benedicks, D. Schnellmann, Whitney-H\"{o}lder continuity of the SRB measure for transversal families of smooth unimodal maps. {\it Invent. Math.}, {\bf 201} (2015), no.3, 773--844.



\bibitem[BV]{BV} L. Barreira, C. Valls, H\"{o}lder Grobman-Hartman linearization. {\it Discrete Contin. Dyn. Syst.}, {\bf18} (2007), no.1, 187--197.



\bibitem[BQ]{BQ} Y.\,Benoist, J.\,Quint, Random walks on reductive groups. {\it Ergebnisse der Mathematik und ihrer Grenzgebiete. 3. Folge. A Series of Modern Surveys in Mathematics} {\bf 62}, Springer, Cham, 2016, xi+323 pp.

\bibitem[BQ16]{BQ16} Y.\,Benoist, J.\,Quint,  Central Limit Theorem for Linear Groups, {\it The Annals of Probability}, {\bf 44} (2016), pp. 1308--1340.

\bibitem[BL]{BL}   P. Bougerol and J. Lacroix, {\it Products of Random Matrices with Applications to Schr\"odinger Operators}, Birkhauser, Boston, 1985.

\bibitem[B]{B} J. Bourgain, On the Furstenberg measure and density of states for the Anderson-Bernoulli model at small disorder, {\it J. Anal. Math.} {\bf 117} (2012), pp. 273–295.

\bibitem[BK]{BK} J. Bourgain, A. Klein, { \it Bounds on the density of states for Schr\"{o}dinger operators}, Invent. Math. {\bf 194} (2013), no.1, pp. 41--72.



\bibitem[BrK]{BrK} M. Brin, Yu. Kifer, Dynamics of Markov chains and stable manifolds for random diffeomorphisms, {\it Ergodic Theory and Dyn. Syst.}, {\bf 7} (1987), pp. 351--374.

\bibitem[BP]{BP} M. Brin, Ya. Pesin, Partially hyperbolic dynamical systems, {\it Izv.Acad.Nauk SSSR ser. mat.} {\bf 1} (1974), pp. 170--212.

\bibitem[BH]{BH}   A.\,Brown,  F.\,Hertz, Measure rigidity for random dynamics on surfaces and related skew products, {\it J. Amer. Math. Soc.} {\bf 30} (2017), pp. 1055--1132.



\bibitem[CK]{CK} M. Campanino, A. Klein, A supersymmetric transfer matrix and differentiability of the density of states in the one-dimensional Anderson model. {\it Comm. Math. Phys.} {\bf 104} (1986), no.2, pp. 227--241.

\bibitem[CS]{CS} W. Craig, B. Simon, Log H\"{o}lder continuity of the integrated density of states for stochastic Jacobi matrices, {\it Comm. Math. Phys.} {\bf 90} (1983), no.2, 207--218.

\bibitem[DG]{DG} D. Damanik, A. Gorodetski, H\"{o}lder continuity of the integrated density of states for the Fibonacci Hamiltonian. {\it Comm. Math. Phys.} {\bf 323} (2013), no.2, pp. 497--515.








\bibitem[Fu1]{Fu} A.\,Furman,  Random walks on groups and random transformations, {\it Handbook of dynamical systems}, vol. 1A, pp. 931--1014, North-Holland, Amsterdam, 2002.

\bibitem[Fu2]{Fu1} A.\,Furman,   What is … a stationary measure? {\it Notices Amer. Math. Soc.} {\bf  58} (2011), no. 9, pp. 1276--1277.







\bibitem[GK]{GK} Z. Gan, H. Krüger, Optimality of log H\"{o}lder continuity of the integrated density of states. {\it Math. Nachr.}, {\bf 284} (2011), no.14-15, 1919–1923.


\bibitem[GM]{GM} I. Gol'dshe\u{\i}d, G. Margulis, Lyapunov exponents of a product of random matrices. {\it Uspekhi Mat. Nauk}, {\bf 44} (1989), no.5, pp. 13–-60.

\bibitem[GS]{GS} M. Goldstein, W. Schlag, H\"{o}lder continuity of the integrated density of states for quasi-periodic Schr\"{o}dinger equations and averages of shifts of subharmonic functions. {\it Ann. of Math.} {\bf 2}, 154 (2001), no.1, 155--203.


\bibitem[GKM]{GKM} A. Gorodetski, V. Kleptsyn, G. Monakov, H\"older regularity of stationary measures, arXiv:2209.12342

\bibitem[GKM2]{GKM2} A. Gorodetski, V. Kleptsyn, G. Monakov, Central Limit Theorem for non-stationary random products of $\SL(2, \R)$ matrices, preprint, arXiv:2411.12003. 

\bibitem[Gu]{Gu} Y. Guivarc'h,  Produits de matrices aleatoires et applications aux proprietes geometriques des sous-groupes du groupe lineaire, {\it Ergodic Theory Dynam. Systems} {\bf  10} (1990), pp. 483--512.

\bibitem[GR]{GR} Y. Guivarc'h, A. Raugi,  Products of random matrices: convergence theorems, {\it Random matrices and their applications (Brunswick, Maine, 1984)}, pp. 31--54, {\it Contemp. Math.}, 50, Amer. Math. Soc., Providence, RI, 1986.



\bibitem[HV]{HV} E. Hart, B. Virág, H\"{o}lder Continuity of the Integrated Density of States in the One-Dimensional Anderson Model. {\it Commun. Math. Phys.} {\bf 355} (2017), pp. 839--863.






\bibitem[Kif1]{Kif1} Yu.\,Kifer, Random dynamics and its applications, {\it Proceedings of the International Congress of Mathematicians}, Vol. II (Berlin, 1998). Doc. Math. 1998, Extra Vol. II, pp. 809--818.

\bibitem[Kif2]{Kif2} Yu.\,Kifer, Ergodic theory of random transformations, {\it Progress in Probability and Statistics,} {\bf 10}, Birkh\"auser Boston, Inc., Boston, MA, 1986. x+210 pp.





\bibitem[L]{L} E. Le Page, Th\'eor\`emes limites pour les produits de matrices al\'eatoires, in: {\it Probability Measures on Groups}, H. Heyer, ed., Springer-Verlag, New York, 1982.

\bibitem[LQ]{LQ} P.\,Liu,  M.\,Qian, Smooth ergodic theory of random dynamical systems, {\it Lecture Notes in Mathematics,} {\bf 1606},  Springer-Verlag, Berlin, 1995, xii+221 pp.

\bibitem[M]{M} D.\,Malicet, Random Walks on $\text{\rm Homeo}(S^1)$, {\it  Commun. Math. Phys.} {\bf 356} (2017), pp. 1083--1116.

\bibitem[MS]{MS} P. March, A. Sznitman, Some connections between excursion theory and the discrete Schr\"{o}dinger equation with random potentials. {\it Probab. Theory Related Fields} {\bf 75} (1987), no.1, pp. 11--53.

\bibitem[MP]{MP} P. Mörters, Yu. Peres, Brownian motion. {\it Camb. Ser. Stat. Probab. Math.}, {\bf 30}, Cambridge University Press, Cambridge, 2010. xii+403 pp.

\bibitem[Mu]{Mu} P. Munger, Frequency dependence of H\"{o}lder continuity for quasiperiodic Schr\"{o}dinger operators. {\it J. Fractal Geom.} {\bf 6} (2019), no.1, pp. 53--65.




\bibitem[ST]{ST} B. Simon, M. Taylor, Harmonic analysis on $\SL(2,\R)$ and smoothness of the density of states in the one-dimensional Anderson model. {\it Comm. Math. Phys.} {\bf 101} (1985), no.1, pp. 1--19.






\bibitem[V]{V} I. Veselić, Existence and regularity properties of the integrated density of states of random Schr\"{o}dinger operators. {\it Lecture Notes in Math.}, {\bf 1917} Springer-Verlag, Berlin, 2008. x+142 pp.

\bibitem[Vi]{Vi} C.\,Villani, Optimal Transport. Old and New,  Grundlehren der mathematischen Wissenschaften [Fundamental Principles of Mathematical Sciences], {\bf 338}, {\it Springer-Verlag, Berlin}, 2009,  xxii + 973 pp.


\end{thebibliography}
\end{document}